\documentclass[reqno]{amsart}

\usepackage[margin=1.5in]{geometry}
\usepackage{palatino}
\usepackage{mathpazo}

\usepackage{amsmath, amssymb, amsthm, mathrsfs}
\usepackage{tikz-cd}
\usepackage{tikz}
\usepackage[colorlinks=true, allcolors=purple]{hyperref}
\usepackage{enumitem}
\setlist[enumerate,1]{label=\textup{(\arabic*)}}
\setlist[enumerate,2]{label=\textup{(\alph*)}}

\newtheorem{theorem}{Theorem}[section]
\newtheorem{lemma}[theorem]{Lemma}
\newtheorem{proposition}[theorem]{Proposition}
\newtheorem{observation}[theorem]{Observation}
\newtheorem{corollary}[theorem]{Corollary}
\newtheorem*{theorem*}{Theorem}
\newtheorem{thmx}{Theorem}

\newtheorem{corx}[thmx]{Corollary}

\theoremstyle{definition}
\newtheorem{definition}[theorem]{Definition}

\newtheorem{example}[theorem]{Example}

\theoremstyle{remark}
\newtheorem{remark}[theorem]{Remark}

\newcommand{\LL}{\mathbb{L}}
\newcommand{\TT}{\mathbb{T}}
\newcommand{\Gm}{\mathbb{G}_m}
\newcommand{\cO}{\mathcal{O}}
\newcommand{\cL}{\mathcal{L}}
\newcommand{\cK}{\mathcal{K}}
\newcommand{\K}{\mathbb{K}}
\newcommand{\DR}{\mathbf{DR}}
\newcommand{\Spec}{\mathrm{Spec\,}}
\newcommand{\dArt}{\mathbf{dArt}}
\newcommand{\Acl}{\mathcal{A}^{2, \mathrm{cl}}}
\newcommand{\bfem}[1]{\textbf{\textit{#1}}}
\newcommand{\R}{{\mathbb R}}

\makeatletter
\g@addto@macro{\endabstract}{\@setabstract}
\newcommand{\authorfootnotes}{\renewcommand\thefootnote{\@fnsymbol\c@footnote}}%
\makeatother

\begin{document}

\title[Transversality and Derived Legendrian
Intersections]{Equivariant Quotients of derived symplectic spaces and \\
Legendrian Intersection Theorem}
\author{Efe \.{I}zbudak}
\author{Kadri \.{I}lker Berktav}

\begin{center}
  \LARGE
  Equivariant Quotients of Derived Symplectic Spaces and
  \\ Legendrian Intersection Theorem \par \bigskip

  \normalsize
  \authorfootnotes
  Efe
  \.{I}zbudak\footnote{efe.izbudak@studium.uni-hamburg.de}\textsuperscript{1},
  Kadri \.{I}lker
  Berktav\footnote{berktav@metu.edu.tr}\textsuperscript{2} \par \bigskip

  \textsuperscript{1}Department of Mathematics, Universit\"{a}t Hamburg\par
  \textsuperscript{2}Department of Mathematics, METU\par \bigskip

\end{center}

\begin{abstract}
  The classical transversality lemma of contact geometry constructs a
  contact structure on a hypersurface transverse to a  Liouville
  vector field using point-set topology and local flows. This paper
  translates the classical transversality lemma into the context of
  derived algebraic geometry and provides the derived Legendrian
  intersection theorem, along with various applications to moduli theory.

  In brief, we first prove that taking the quotient of a derived
  symplectic space descends the symplectic data to a contact
  structure, avoiding the transverse hypersurface requirement, such
  that the fundamental vector field of a weight 1
  $\mathbb{G}_m$-action in the derived setting replaces the classical
  Liouville vector field.  Second, we establish the derived
  Legendrian intersection theorem using base change, an
  $\infty$-categorical descent cube, and $\Gm$-equivariant lifts
  along the symplectification projection.

  As applications of the main results, we first examine the derived
  geometry of the discriminant loci of 1-jet bundles and show that
  these loci carry a $(-1)$-shifted contact structure. In addition,
  we show that our results apply to certain moduli problems,
  including projective Higgs bundles, $\ell$-adic local systems, and
  Lie 2-groups, and produce further examples of contact derived moduli stacks.
\end{abstract}
{\small\tableofcontents}

\section{Introduction and Summary}

A fundamental lemma in classical differential geometry \cite[Lemma
1.4.5]{Geiges} states that a Liouville vector field on a symplectic
manifold induces a contact structure on any transverse hypersurface.
This classical formulation  relies on smooth local flows, differential forms,
and point-set transversality \cite[Section
1.4]{Geiges}. These methods do not apply to moduli
problems in algebraic geometry, where spaces often have
singularities, nilpotence, and non-trivial automorphism groups.
Similar drawbacks emerge for the intersection problems as well.
To address these issues,  nowadays folklore suggests, for instance,
the use of shifted geometric (e.g., symplectic) structures,
equivariant group actions, exact triangles, and derived cotangent
complexes, or in short, the framework of Derived Algebraic Geometry (DAG).

DAG provides a setting for non-generic situations in geometry, such
as non-transversal intersections and ``bad'' quotients, and it offers
generalized versions of some familiar geometric structures. We adopt
this framework in this paper.
\medskip

\paragraph{\bfem{Conventions and notations}.} Throughout the paper, $
\mathbb{K} $ will be an algebraically closed field of characteristic
zero. All cdgas will be graded in nonpositive degrees and  over
$\mathbb{K}.$ We always consider $\K$-schemes/stacks, and we assume
that all classical $ \K $-schemes are \textit{locally of finite
type}, and that all derived $ \K $-schemes/stacks $ {X} $ are
\textit{locally finitely presented.} The only exception is
Section~\ref{subsec:ladic}, which concerns $\ell$-adic local systems.
\bigskip

\paragraph{\bfem{Our results.}} We now summarize our main results,
leaving full explanations and
definitions to the main text. Our first result shows that the
classical transversality lemma of contact geometry can be
reformulated in the derived context. More specifically, using the
geometry of $\Gm$-torsors
and symplectification in the sense of \cite[Section~4]{kib1},
\cite[Theorem~3.9]{kib2}, we prove the following result.

\begin{thmx}[cf.~Theorem~\ref{thm:derived transversality}]\label{thm:A}
  Given an $n$-shifted symplectic derived Artin stack $(\widetilde{X},
  \omega_{\widetilde{X}})$
  equipped with a $\Gm$-action of weight 1 such that $\widetilde{X}
  \to [\widetilde{X}/\Gm]$ is trivial on an \'etale cover,
  the stack quotient $[\widetilde{X} / \Gm]$ inherits an
  $n$-shifted contact structure.
\end{thmx}

This theorem formalizes \textit{derived symplectification}. Taking the
quotient of a homogeneous symplectic space descends the symplectic
data to a contact structure, avoiding
a transverse hypersurface. In the derived setting, the Liouville
vector field corresponds to the
fundamental vector field of a weight 1 $\Gm$-action, and the
contact manifold corresponds to the stack quotient. Classical coordinate
arguments are replaced by stable $\infty$-categorical operations.
This relies on generalized quotient groupoids,
shifted symplectic forms \cite[Definition 1.18]{PTVV},
and exact triangles of perfect complexes.
\medskip

The second main result of this paper concerns the \textit{Legendrian
intersection problem}, which has no direct classical counterpart. In
brief, the theory above extends to Legendrian morphisms and gives
rise to an affirmative answer to the Legendrian intersection problem
in the context of derived geometry.
Derived contact geometry links Legendrian structures to Lagrangian
structures via $\Gm$-equivariant lifts along the symplectification projection.
Lifting intersection data to the symplectic level yields the following result.

\begin{thmx}[cf.~Theorem~\ref{thm:Leg intersection}]\label{thm:B}
  Intersection of two Legendrians $f_1 \colon L_1 \to X$ and $f_2
  \colon L_2 \to X$ in an $n$-shifted contact derived Artin stack admits
  an $(n-1)$-shifted  contact structure.
\end{thmx}
This lifting property translates intersection problems in contact
moduli to equivariant intersection problems in symplectic moduli.
This establishes the desired shifted contact structure on the derived
intersection of two Legendrian stacks.
As an interesting consequence of Theorem \ref{thm:B}, we then obtain:

\begin{corx}[cf.~Section~\ref{sec: (-1)-shifted contact via derived
  leg int}]\label{cor:B}
  The \bfem{derived discriminant locus}
  $ \Delta\mathrm{loc}(f):= L_F \times_{J^1(L)}^h L_0$ of a regular
  function $f$ on a smooth $\K$-scheme $L$ admits a $(-1)$-shifted
  contact structure, where the map $F:p \mapsto (f(p),d_{dR}f_p)$
  defines a Legendrian embedding $L_F$ of $L$ into the 1-jet space
  $(J^1(L), \xi_{jet})$ and $L_0$ denotes the Legendrian given by the
  zero-section morphism.
\end{corx}

Last but not least, we also provide several applications of Theorem
\ref{thm:A} and Theorem \ref{thm:B} to moduli theory and give
examples of derived stacks with a shifted contact structure.

\begin{corx}\label{cor:C}
  Our results apply to the following moduli problems.
\begin{enumerate}\itemsep=0.1cm
    \item Let $X$ be a smooth proper variety of dimension $d$ over
      $\K$ and $G$ be a reductive algebraic group. Then the derived
      moduli stack $P\mathrm{Higgs}_G(X)$ of projective Higgs bundles
      admits a canonical $(1-d)$-shifted contact structure in the
      sense of the symplectification of Remark~\ref{rmk:etale} (cf.
      Corollary \ref{cor:projective_higgs}).

      Restricting to a smooth proper curve $C$, the derived
      intersection of projectivized irreducible components of the
      global nilpotent cone, once they carry Legendrian structures,
      within the $0$-shifted ambient contact stack
      $P\mathrm{Higgs}_G(C)$ yields a strictly derived moduli space
      equipped with a canonical $(-1)$-shifted contact structure (cf.
      Corollary \ref{cor:nilpotent_intersection}).
    \item Let $X$ be a smooth projective geometrically connected
      variety of dimension $d$ over $\mathbb{F}_q$, and let $\bar X
      := X \times_{\mathbb{F}_q} \bar{\mathbb{F}}_q$. The Frobenius
      multiplies the $(2-2d)$-shifted symplectic structure of the
      derived moduli stack $LocSys_{\ell,r}(\bar X)$ of $\ell$-adic
      local systems by $q^d$. The quotient $[LocSys_{\ell,r}(\bar
      X)/\mathbb{Z}]$ by Frobenius inherits a canonical
      $(2-2d)$-shifted symplectic structure twisted by the line
      bundle of this character (cf. Proposition
      \ref{prop:discrete_descent} and Corollary \ref{cor: l-adic local sys}).
    \item The derived stack quotient $\mathbb{P}(T^*[2]BG) :=
      [T^*[2]BG / \Gm]$ carries a canonical 2-shifted contact
      structure in the sense of the symplectification of
      Remark~\ref{rmk:etale} (cf. Corollary \ref{cor: Lie 2-group}).
    \item For a closed subgroup $H \subset G$, denote by
      $\mathcal{N}^*_{BH/BG}[2]$ the 2-shifted conormal stack. Then
      the morphism $\mathbb{P}(\mathcal{N}^*_{BH/BG}[2]) \to
      \mathbb{P}(T^*[2]BG)$ is Legendrian in the sense of the
      symplectification (Remark~\ref{rmk:etale}). Furthermore, the
      derived fiber product of two such Legendrians associated with
      subgroups $H_1$ and $H_2$ carries a canonical 1-shifted contact
      structure (cf. Corollary \ref{cor: intersection of prequan}).
  \end{enumerate}
\end{corx}

\paragraph{\bfem{Organization of the paper.}}  Sections
\ref{sec:stable_cats} through
\ref{sec:structures} provide background material on stable
$\infty$-categories, derived commutative algebra, cotangent
complexes, and shifted symplectic and contact structures. Section
\ref{sec:classical} overviews the classical setup of transversality
in contact geometry. Section
\ref{sec:setup} formulates the derived setup and the first main
theorem of this paper (Theorem \ref{thm:derived transversality}). In
Section \ref{sec:proof}, we give the proof of Theorem \ref{thm:A} and
analyze the failure of classical transversality.  Section
\ref{sec:legendrian} establishes Theorem \ref{thm:B}. Finally, in
Section \ref{sec:moduli}, we prove Corollary \ref{cor:B} and
elaborate on the statements of Corollary \ref{cor:C}, giving the proof of each.

\section{Recollection}

\subsection{Elements of Stable
\texorpdfstring{$\infty$}{infinity}-Category Theory} \label{sec:stable_cats}

In derived geometry, sets and sheaves are
replaced by stable $\infty$-categories. Let $X$ be a topological
space or a scheme.

\begin{definition}[Stable $\infty$-Categories]
  Following Lurie \cite[Definition 1.1.1.9 and Proposition
  1.1.3.4]{LurieHA}, a \bfem{stable
  $\infty$-category} is an
  $\infty$-category $\mathcal{C}$ possessing a zero object, as well
  as all finite limits and colimits, such that a
  commutative square in $\mathcal{C}$ is a pullback if and only if it
  is a pushout.

  Here, the suspension functor $\Sigma: \mathcal{C}
  \xrightarrow{\sim} \mathcal{C}$
  is an equivalence, allowing every morphism to be
  completed to a fiber sequence.
\end{definition}

\begin{definition}[The Derived $\infty$-Category]
  For a derived stack $X$, the $\infty$-category of quasi-coherent
  sheaves is denoted $L_{qcoh}(X)$ \cite[Section 1.1]{PTVV}. Its
  homotopy category is the quasi-coherent derived category
  $D_{qcoh}(X)$. $L_{qcoh}(X)$ is a stable, symmetric monoidal
  $\infty$-category possessing internal mapping objects
  $\mathbb{R}\underline{\mathcal{H}om}(E, F)$.
\end{definition}

This localization
resolves the classical failure of base change for non-flat morphisms,
permitting the
global computation of
derived functors.

\begin{definition}[Internal Mapping Objects and Mapping Spaces]
  In a stable, symmetric monoidal $\infty$-category $\mathcal{C}$
  equipped with a closed structure, the internal mapping object
  $\mathbb{R}\underline{\mathcal{H}om}(E, F)$ resides within
  $\mathcal{C}$. The \bfem{mapping space}
  $\mathrm{Map}_{\mathcal{C}}(E, F)$ is a Kan complex in the
  $\infty$-category of spaces $\mathcal{S}$. By adjunction, morphisms
  in $\mathcal{C}$ are
  identified with global sections of the internal mapping object.
\end{definition}

\begin{definition}[Homotopy Fibers and Fiber Sequences]
  Let $f: E \to F$ be a morphism in the stable $\infty$-category $L_{qcoh}(X)$.
  The \bfem{homotopy fiber} $\mathrm{fib}(f)$ is defined
  as the limit
  (homotopy pullback) of the diagram $E \xrightarrow{f} F \leftarrow
  0$.

  This construction
  gives a \bfem{fiber sequence} $\mathrm{fib}(f) \to E
  \xrightarrow{f} F$, which
  serves as the derived replacement for a short exact sequence
  \cite[Section 1.1.2]{LurieHA}.
\end{definition}

In a stable $\infty$-category, fiber sequences and cofiber sequences
coincide up to a shift. This endows its homotopy category with a
triangulated structure, allowing morphisms to be completed to
an exact sequence. The homotopy cofiber provides
a method for computing homological obstructions.

\begin{remark}
  Fiber sequences replace short exact sequences in derived geometry.
  Classically, the kernel or cokernel of a map of vector bundles on a
  singular space might fail to be a vector bundle. The homotopy cofiber
  replaces the cokernel, and the homotopy fiber
  replaces the kernel. Unlike kernels and cokernels, homotopy
  fibers and cofibers preserve
  dualizability.
\end{remark}

\begin{lemma}[Pasting Law for Pullbacks]
  \label{lem:pullback_pasting}
  Let $\mathcal{C}$ be an $\infty$-category and suppose we are given
  a diagram $\Delta^1 \times \Delta^2 \to \mathcal{C}$ depicted as
  \begin{align*}
    \begin{tikzcd}[ampersand replacement=\&]
      A \arrow[r] \arrow[d] \& B \arrow[r] \arrow[d] \& C \arrow[d] \\
      X \arrow[r] \& Y \arrow[r] \& Z
    \end{tikzcd}
  \end{align*}
  If the right square is a pullback in $\mathcal{C}$, then the left
  square is a pullback if and only if the outer rectangle is a pullback.
\end{lemma}
\begin{proof}
  This follows by applying the pasting law for pushouts \cite[Lemma
  4.4.2.1]{LurieHTT} to the opposite $\infty$-category
  $\mathcal{C}^{op}$. A pullback square in $\mathcal{C}$ is a pushout
  square in $\mathcal{C}^{op}$. The diagram $\Delta^1 \times \Delta^2
  \to \mathcal{C}$ corresponds to a diagram $(\Delta^1 \times
  \Delta^2)^{op} \simeq \Delta^1 \times \Delta^2 \to
  \mathcal{C}^{op}$. Applying \cite[Lemma 4.4.2.1]{LurieHTT} in
  $\mathcal{C}^{op}$ finishes the proof.
\end{proof}

\begin{lemma}[Exact Triangles of Fibers and the Pasting Law]
  \label{lem:pasting_law}
  Let $\mathcal{C}$ be a stable $\infty$-category. Given composable
  morphisms $X \xrightarrow{f} Y \xrightarrow{g} Z$ in $\mathcal{C}$,
  the homotopy fibers assemble into a fiber sequence
  \begin{align*}
    \mathrm{fib}(f) \to \mathrm{fib}(g \circ f) \to \mathrm{fib}(g).
  \end{align*}
\end{lemma}
\begin{proof}
  Since $\mathcal{C}$ is a stable $\infty$-category, it admits all
  finite limits \cite[Proposition 1.1.3.4]{LurieHA}.

  \noindent The composition $g \circ f$ determines a  morphism of
  cospans from $(X \xrightarrow{g \circ f} Z \leftarrow 0)$ to $$(Y
  \xrightarrow{g} Z \leftarrow 0).$$

  The functoriality of limits extends this to a diagram $\Delta^1
  \times \Delta^2 \to \mathcal{C}$ containing two adjacent squares
  \begin{align*}
    \begin{tikzcd}[ampersand replacement=\&]
      \mathrm{fib}(g \circ f) \arrow[r] \arrow[d] \& \mathrm{fib}(g)
      \arrow[r] \arrow[d] \& 0 \arrow[d] \\
      X \arrow[r, "f"] \& Y \arrow[r, "g"] \& Z
    \end{tikzcd}
  \end{align*}
  The right square is a pullback by the definition of the homotopy
  fiber $\mathrm{fib}(g)$ \cite[Definition 1.1.1.6]{LurieHA}.
  The outer rectangle is a pullback by the definition of the homotopy
  fiber $\mathrm{fib}(g \circ f)$.

  Lemma \ref{lem:pullback_pasting} implies the left square
  \begin{align*}
    \begin{tikzcd}[ampersand replacement=\&]
      \mathrm{fib}(g \circ f) \arrow[r] \arrow[d] \& \mathrm{fib}(g)
      \arrow[d] \\
      X \arrow[r, "f"] \& Y
    \end{tikzcd}
  \end{align*}
  is a pullback square.
  We then extend the diagram upward by defining $F$ as the homotopy
  fiber of the morphism $\mathrm{fib}(g \circ f) \to
  \mathrm{fib}(g)$, obtaining the diagram
  \begin{align*}
    \begin{tikzcd}[ampersand replacement=\&]
      F \arrow[r] \arrow[d] \& 0 \arrow[d] \\
      \mathrm{fib}(g \circ f) \arrow[r] \arrow[d] \& \mathrm{fib}(g)
      \arrow[d] \\
      X \arrow[r, "f"] \& Y
    \end{tikzcd}
  \end{align*}
  The top square is a pullback by the definition of $F$. The middle
  square is a pullback as demonstrated above.

  Applying Lemma \ref{lem:pullback_pasting} again to this vertical
  arrangement, the outer rectangle
  \begin{align*}
    \begin{tikzcd}[ampersand replacement=\&]
      F \arrow[r] \arrow[d] \& 0 \arrow[d] \\
      X \arrow[r, "f"] \& Y
    \end{tikzcd}
  \end{align*}
  is a pullback square. By the definition of the homotopy fiber, we
  obtain an equivalence $F \simeq \mathrm{fib}(f)$.

  The definition of $F$ as the fiber of $\mathrm{fib}(g \circ f) \to
  \mathrm{fib}(g)$ provides the fiber sequence
  \begin{align*}
    F \to \mathrm{fib}(g \circ f) \to \mathrm{fib}(g).
  \end{align*}
  Since $F \simeq \mathrm{fib}(f)$, we obtain the fiber sequence
  \begin{align*}
    \mathrm{fib}(f) \to \mathrm{fib}(g \circ f) \to \mathrm{fib}(g).
  \end{align*}
\end{proof}

\begin{definition}[Perfect Complexes]
  A \bfem{perfect complex} on a derived stack $X$ is
  defined as a \textbf{dualizable object} in the symmetric monoidal
  $\infty$-category $L_{qcoh}(X)$ \cite[Section 1.1]{PTVV}. This
  means that $E^\vee := \mathbb{R}\underline{\mathcal{H}om}(E,
  \mathcal{O}_X)$ satisfies $E^\vee
  \otimes^{\mathbb{L}}_{\mathcal{O}_X} F \simeq
  \mathbb{R}\underline{\mathcal{H}om}(E, F)$ naturally in $F$. In
  particular, the biduality map $E \to (E^\vee)^\vee$ is an equivalence.
\end{definition}

Perfect complexes represent the dualizable objects in the derived
category of quasi-coherent sheaves. They provide the geometric
finiteness condition required in shifted symplectic geometry
\cite[Section 1.1]{Calaque2019}. Without this
condition, the biduality map fails to be an equivalence, precluding the
definition of a non-degenerate symplectic form.

\subsection{Derived Commutative Algebras and Intersections}
\label{sec:derived_comm}

\begin{definition}[Derived Tensor Product]
  Let $M, N \in L_{qcoh}(X)$. The \bfem{derived tensor product} $M
  \otimes^{\mathbb{L}}_{\cO_X} N$
  is the canonical symmetric monoidal structure on the stable
  $\infty$-category $L_{qcoh}(X)$ \cite[Section 4.5.2]{LurieHA}.
  Its cohomology group in degree $-i$ computes the higher derived functor
  $\mathrm{Tor}_i^{\cO_X}(M, N)$.
\end{definition}

The derived tensor product is the algebraic
operation for derived intersection theory. Computing intersections
via derived tensor products preserves geometric intersection
multiplicities \cite[Chapitre V, \S C.1]{SerreAlgLoc}. This
extends the intersection theory of Serre by
embedding intersection multiplicities into the structure sheaf.

\begin{definition}[Commutative Differential Graded Algebra]
  A \bfem{commutative differential graded algebra} (cdga) over $\K$
  is a graded $\K$-algebra $A = \bigoplus_{i \le 0} A^i$ concentrated
  in non-positive cohomological degrees, equipped with a differential
  $d: A^i \to
  A^{i+1}$ satisfying $d^2 = 0$ and the graded Leibniz rule.
  The $\infty$-category of such algebras is denoted by
  $\mathrm{cdga}_\K^{\leq 0}$.
\end{definition}
These algebras generalize commutative rings by
encoding relations and syzygies into the differential
structure. The study of affine derived schemes is formally dual to
the study of these commutative differential graded algebras,
providing an algebraic framework for derived geometry.

\begin{definition}[Homotopy Pullback]
  Let $X \to Z$ and $Y \to Z$ be morphisms of derived spaces. The
  \bfem{homotopy pullback} or derived intersection $X
  \times_Z^{\mathbb{R}} Y$ is the limit completing the
  pullback diagram in the $\infty$-category of derived spaces.

  Affinely, if
  $X = \Spec(A)$, $Y = \Spec(B)$, and $Z = \Spec(C)$, the derived
  intersection corresponds to the pushout $A
  \otimes_C^{\mathbb{L}} B$ computed within
  $\mathrm{cdga}_\K^{\leq 0}$
  \cite[Section 2.2.2]{ToenHAG}.
\end{definition}

The universal property of the homotopy pullback ensures fiber
products of derived stacks exist. This framework resolves classical
pathologies where fiber products have unexpected dimensions, since
derived fiber products have the expected virtual dimension. The
resulting space is a derived scheme
capturing the intersection geometry.

\begin{remark}
  Let $X$ and $Y$ be closed subschemes of a scheme $Z$. The classical
  intersection $X \times_Z Y$ has the structure sheaf $\mathcal{O}_X
  \otimes_{\mathcal{O}_Z} \mathcal{O}_Y$. In general, for instance
  when $X$ and $Y$ do not intersect properly, the higher Tor sheaves
  $\mathrm{Tor}_i^{\mathcal{O}_Z}(\mathcal{O}_X, \mathcal{O}_Y)$ with
  $i > 0$ need not vanish, and the tensor product $\mathcal{O}_X
  \otimes_{\mathcal{O}_Z} \mathcal{O}_Y$ discards them. The classical
  intersection then no longer determines the intersection
  multiplicity, which is the alternating sum of the lengths of all
  these Tor sheaves \cite[Chapitre V, \S C.1]{SerreAlgLoc}.
\end{remark}

\begin{observation}
  \label{prop:cdga_necessity}
  The derived fiber product of the origin $\mathrm{Spec}(\mathbb{K})$
  with itself over the affine line $\mathbb{A}^1_{\mathbb{K}}$ is
  represented by the commutative differential graded algebra
  $\mathbb{K}[\epsilon]/(\epsilon^2)$ with $\epsilon$ in degree $-1$.
  Its homology records the excess of this intersection, whose virtual
  dimension is $0 + 0 - 1 = -1$.
\end{observation}
\begin{proof}
  The classical fiber product $\mathrm{Spec}(\mathbb{K})
  \times_{\mathrm{Spec}(\mathbb{K}[x])} \mathrm{Spec}(\mathbb{K})$ is
  computed by the pushout in the category of commutative rings
  \begin{align*}
    \mathbb{K} \otimes_{\mathbb{K}[x]} \mathbb{K} \cong \mathbb{K}.
  \end{align*}
  This pushout evaluates to $\mathbb{K}$ and discards this excess.

  The derived intersection is governed by the derived tensor product
  $\mathbb{K} \otimes_{\mathbb{K}[x]}^{\mathbb{L}} \mathbb{K}$ in the
  $\infty$-category of commutative differential graded algebras. We
  compute this derived tensor product by replacing $\mathbb{K}$ with
  the Koszul resolution $K(\mathbb{K}[x], x)$ over $\mathbb{K}[x]$
  \begin{align*}
    0 \to \mathbb{K}[x] \xrightarrow{x} \mathbb{K}[x] \to 0
  \end{align*}
  where the terms are concentrated in cohomological degrees $-1$ and
  $0$. Applying the functor $- \otimes_{\mathbb{K}[x]} \mathbb{K}$ to
  the Koszul resolution we obtain the chain complex
  \begin{align*}
    0 \to \mathbb{K} \xrightarrow{0} \mathbb{K} \to 0.
  \end{align*}
  The homology of this complex computes the higher derived functors
  \begin{align*}
    \mathrm{Tor}_0^{\mathbb{K}[x]}(\mathbb{K}, \mathbb{K}) &\cong \mathbb{K} \\
    \mathrm{Tor}_1^{\mathbb{K}[x]}(\mathbb{K}, \mathbb{K}) &\cong \mathbb{K}.
  \end{align*}
  The derived intersection possesses a structure sheaf represented by
  the commutative differential graded algebra
  $\mathbb{K}[\epsilon]/(\epsilon^2)$ with generator $\epsilon$ in
  degree $-1$ and differential $d(\epsilon) = 0$. The
  $\infty$-categorical theory preserves the homological data of the
  non-transverse intersection \cite[Section 2]{Toen2014}.
\end{proof}

\subsection{From Classical Differentials to Derived Stacks} \label{sec:diffs}

\begin{definition}[K\"{a}hler Differentials]
  Let $A$ be a commutative $R$-algebra. The module of
  \bfem{K\"{a}hler differentials} $\Omega^1_{A/R}$ is the
  $A$-module generated by symbols $da$ subject to the
  Leibniz rule. For singular spaces its conormal sequence fails to be
  left exact (Observation~\ref{prop:cotangent_necessity}), which
  makes it inadequate for derived moduli theory.
\end{definition}

The module of K\"{a}hler differentials is locally free for
smooth schemes, but behaves poorly for singularities.
Its lack of exactness is remedied by the cotangent complex, which
Illusie developed for infinitesimal deformation problems
\cite[Introduction]{Illusie}. The
failure of the cotangent functor to be left exact means the
classical sheaf cannot distinguish between singular
structures with identical tangent spaces.

In what follows, we present the underlying (higher) spaces of
interest in this paper and outline their essential properties, along
with several useful remarks.

\begin{definition}[Affine Derived Schemes]
  The $\infty$-category of \bfem{affine derived schemes} over $\K$,
  denoted $\mathbf{dAff}_\K$, is defined as the opposite
  $\infty$-category $(\mathrm{cdga}_\K^{\leq 0})^{\mathrm{op}}$. For
  any $A \in \mathrm{cdga}_\K^{\leq 0}$, the corresponding affine
  derived scheme is denoted $\Spec(A)$ \cite[Section 2.2]{Toen2014}.
\end{definition}

\begin{definition}[Derived Stacks]
  A \bfem{derived stack} over $\K$ is a functor $X \colon
  \mathrm{cdga}_\K^{\leq 0} \to \mathcal{S}$ taking values in the
  $\infty$-category of spaces (Kan complexes), such that $X$
  satisfies exact hyperdescent with respect to the \'{e}tale
  topology. The $\infty$-category of derived stacks, denoted
  $\mathbf{dSt}_\K$, is constructed as the left Bousfield
  localization of the $\infty$-category of presheaves
  $\mathrm{Fun}(\mathrm{cdga}_\K^{\leq 0}, \mathcal{S})$ along
  \'{e}tale hypercovers \cite[Definition 2.2.2.14]{ToenHAG}.
\end{definition}
In other words, a derived stack $ {X} $ is  a  functor  $$ {X}:
\mathrm{cdga}_{\K}^{\leq 0} \rightarrow  \mathcal{S}, \ \ A\mapsto
X(A) \simeq \mathrm{Map}_{\mathbf{dSt}_\K}(\Spec(A), {X}),$$
satisfying a descent condition.
For more details, we refer to \cite[Section 2.2.2]{ToenHAG}.
It should then be noted that any affine derived $\mathbb{K}$-scheme
$X$ can be corepresented by a cdga $A$, where $\Spec(A)$ is given as
the functor \[ \Spec (A): B\in \mathrm{cdga}_\K^{\leq 0}  \longmapsto
Hom_{\mathrm{cdga}_{\mathbb{K}}^{\leq 0}} (A,B) \in \mathcal{S}.\]

The concept of a derived stack may sometimes be too wild for some
purposes because a typical $A$-point $p\in X(A)$ is represented by a
general morphism $p: \Spec A \to X$ in $\mathbf{dSt}_\K$. When we ask
for additional conditions on those morphisms, such as being a Zariski
open embedding or being smooth (of some relative dimension), the
resulting space may admit a tractable description in terms of affine
covers, leading to the following special kind of derived stacks.

\begin{definition}[Derived Schemes]
  \label{def:derived_schemes}
  An object $X$ in the $\infty$-category of derived stacks
  $\mathbf{dSt}_\K$ is a \bfem{derived scheme} over $\K$ if it can be
  covered by Zariski open affine derived schemes $Y \subset X$. The
  $\infty$-category of derived schemes, denoted $\mathbf{dSch}_\K$,
  is the full subcategory of $\mathbf{dSt}_\K$ spanned by the derived
  schemes \cite[Definition 2.2]{kib2}.
\end{definition}

\begin{definition}[Smooth Morphisms of Derived Stacks]
  A morphism of derived Artin stacks $f: X \to Y$ is \bfem{smooth} if
  it is smooth in the sense of \cite[Section 2.2.3]{ToenHAG}. When
  $Y$ is representable, the relative cotangent complex $\LL_{X/Y}$ of
  a smooth morphism is then perfect and its dual $\TT_{X/Y}$ is
  connective \cite[Lemma 2.2.5.2]{ToenHAG}. For instance, $BG \to
  \Spec \K$ is smooth for a smooth affine group scheme $G$, and
  $\LL_{BG} \simeq \mathfrak{g}^\vee[-1]$ is not concentrated in
  degree $0$ \cite[Section 1.2, Classifying stacks]{PTVV}.
\end{definition}

\begin{definition}[Derived Artin Stacks]
  A \bfem{derived Artin stack} is a derived stack
  which is $m$-geometric for some integer $m$, for the \'{e}tale
  topology, and for the class of smooth maps \cite[Section~2.3.1]{kib2}.
\end{definition}
As mentioned before, all derived $ \K $-schemes/stacks $ {X} $ are
assumed to be  \textit{locally finitely presented.} Let us also
review key properties of cotangent and tangent complexes associated
to these higher spaces.
\begin{definition}[Trivial Square-Zero Extensions and Derived Derivations]
  Let $X$ be a derived stack and $M \in L_{qcoh}(X)$. The
  \bfem{trivial square-zero extension} of $X$ by $M$, denoted
  $X[M]$, is defined as the relative derived spectrum
  $\mathbb{R}\mathbf{Spec}_X(\cO_X \oplus M)$, where $M$ acts as a
  square-zero ideal. The space of \bfem{derived derivations}
  $\mathrm{Der}(X, M)$ is the space
  $\mathrm{Map}_{X/\mathbf{dSt}_\K}(X[M], X)$ of retractions of the
  zero section $X \to X[M]$ \cite[Definition 1.3.3]{CPTVV}.
\end{definition}

\begin{definition}[Cotangent and Tangent Complexes]
  For a derived Artin stack $X$, assumed to be locally of finite
  presentation over $\K$, the \bfem{cotangent complex}
  $\LL_{X/\K} \in L_{qcoh}(X)$ is the complex defined by
  corepresenting derived derivations. It satisfies the adjunction
  $\mathrm{Map}_{L_{qcoh}(X)}(\LL_{X/\K}, M) \simeq
  \mathrm{Map}_{X/\mathbf{dSt}_\K}(X[M], X)$ \cite[Section
  1.3.1]{CPTVV}. Since $X$ is locally of finite presentation,
  $\LL_{X/\K}$ is a perfect complex and thus dualizable. Its dual is the
  \bfem{tangent complex} $\TT_{X/\K} \simeq
  \mathbb{R}\underline{\mathcal{H}om}(\LL_{X/\K}, \cO_X)$, which
  classifies derived vector fields.
\end{definition}

The cotangent complex provides a derived replacement for
K\"{a}hler differentials. By extending classical differentials
into a perfect complex, it captures the obstruction theory of the
space \cite[Introduction]{Illusie}. The higher cohomology groups
contain the obstruction classes for deformations, while
the degree zero part recovers the classical differentials.

\begin{observation}
  \label{prop:cotangent_necessity}
  Let $A = R/I$ be a quotient of a commutative ring $R$. The module
  of K\"{a}hler differentials $\Omega^1_A$ gives a right-exact
  conormal sequence. The derived cotangent complex $\mathbb{L}_A$
  lifts this to an exact triangle in the derived category of $A$-modules.
\end{observation}
\begin{proof}
  The classical module of K\"{a}hler differentials fits into the
  conormal sequence
  \begin{align*}
    I/I^2 \to \Omega^1_R \otimes_R A \to \Omega^1_A \to 0.
  \end{align*}
  This sequence is right-exact. The left morphism maps a class
  $\overline{f} \in I/I^2$ to $df \otimes 1$. If $A$ is singular,
  this morphism is generally not injective. The kernel corresponds to
  the relations among the generators of $I$.

  The derived cotangent complex $\mathbb{L}_A$ extends this sequence.
  It is constructed by taking a simplicial resolution of $A$ by free
  commutative algebras over the base ring and applying the K\"{a}hler
  differentials functor degree-wise, and $\mathbb{L}_{A/R}$ is
  constructed in the same way over $R$. The ring homomorphism $R \to
  A$ induces the transitivity sequence, which is an exact triangle in
  the derived category of $A$-modules
  \begin{align*}
    \mathbb{L}_R \otimes_R^{\mathbb{L}} A \to \mathbb{L}_A \to \mathbb{L}_{A/R}.
  \end{align*}
  Assume $A$ is a global complete intersection. The relative
  cotangent complex $\mathbb{L}_{A/R}$ is quasi-isomorphic to the
  shifted conormal module $(I/I^2)[1]$ \cite[Chapitre III,
  Proposition 3.2.4(iii)]{Illusie}. We apply the shift functor $[-1]$
  to rotate the exact triangle and obtain
  \begin{align*}
    \mathbb{L}_{A/R}[-1] \to \mathbb{L}_R \otimes_R^{\mathbb{L}} A
    \to \mathbb{L}_A.
  \end{align*}
  Substituting the quasi-isomorphism $\mathbb{L}_{A/R}[-1] \simeq
  I/I^2$ results in the exact triangle
  \begin{align*}
    I/I^2 \to \mathbb{L}_R \otimes_R^{\mathbb{L}} A \to \mathbb{L}_A.
  \end{align*}
  Applying the homology functor to this triangle produces a long
  exact sequence. The zeroth homology modules are $H_0(\mathbb{L}_R
  \otimes_R^{\mathbb{L}} A) \cong \Omega^1_R \otimes_R A$ and
  $H_0(\mathbb{L}_A) \cong \Omega^1_A$. The long exact sequence terminates as
  \begin{align*}
    H_1(\mathbb{L}_A) \to I/I^2 \to \Omega^1_R \otimes_R A \to \Omega^1_A \to 0.
  \end{align*}
  This sequence recovers the classical conormal sequence and extends
  it to the left. The module $H_1(\mathbb{L}_A)$ measures the failure
  of injectivity.

\end{proof}
We now give several definitions, which will be central for the rest
of the paper.
\begin{definition}
  \label{def:derived_group_scheme}
  Let $S$ be a base derived scheme. A \bfem{derived group scheme}, or
  simply a \bfem{group scheme} over $S$ is a functor $G \colon
  \mathrm{cdga}_{\mathbb{K}}^{\leq 0} \to \mathcal{S}$ that factors
  through $\mathrm{Grp}(\mathcal{S})$, which denotes the
  $\infty$-category of group objects in $\mathcal{S}$, such that the
  underlying functor to $\mathcal{S}$ is representable by a derived
  scheme $X$ over $S$.

  If $X$ is affine, then $G$ is a \bfem{derived affine group scheme}.
  If the structure map $X\to S$ is smooth then $G$ is a \bfem{smooth
  derived group scheme}. \cite[adapted from App. A.2]{Iwanari}.
\end{definition}

We will usually confuse the functor $G$ with the derived scheme $X$
in the above definition of a derived group scheme since by the
derived Yoneda lemma the $\infty$-category of derived schemes embeds
fully faithfully into the $\infty$-category of presheaves of spaces.

\begin{definition}
  \label{def:multiplicative_group}
  Let $S = \Spec(\K)$. The \bfem{multiplicative group scheme}
  $\mathbb{G}_m$ over $S$ is the smooth derived affine group scheme
  represented by the commutative differential graded algebra $\K[t,
  t^{-1}]$ concentrated in degree zero.

  The group scheme structure is determined functorially. For any $A
  \in \mathrm{cdga}_{\K}^{\leq 0}$, the space $\mathbb{G}_m(A)$ is
  the space of units of $A$, whose set of components is the group
  $H^0(A)^\times$ \cite[adapted from Section 1.1]{Waterhouse}.
\end{definition}

\begin{definition}
  \label{def:quotient_groupoid}
  Let $G$ be a smooth group scheme over a base derived scheme $S$.
  Assume that $G$ acts on a derived scheme $X$ over $S$. The
  \bfem{quotient groupoid} $B_\bullet(X, G)$ is the simplicial object
  with degree $n$ terms $X \times_S G^n$.

  The face maps $d_i \colon X \times_S G^n \to X \times_S G^{n-1}$
  and degeneracy maps $s_i \colon X \times_S G^n \to X \times_S
  G^{n+1}$ are defined on points $(x, g_1, \dots, g_n)$ by
  \begin{align*}
    d_0(x, g_1, \dots, g_n) &= (g_1 \cdot x, g_2, \dots, g_n) \\
    d_i(x, g_1, \dots, g_n) &= (x, g_1, \dots, g_{i+1} g_i, \dots,
    g_n) \quad 0 < i < n \\
    d_n(x, g_1, \dots, g_n) &= (x, g_1, \dots, g_{n-1}) \\
    s_i(x, g_1, \dots, g_n) &= (x, g_1, \dots, g_i, e, g_{i+1},
    \dots, g_n) \quad 0 \leq i \leq n
  \end{align*}
  where $g_1 \cdot x$ denotes the action of $G$ on $X$ and $e$ is the
  identity of $G$. The object $B_\bullet(X, G)$ is a smooth groupoid
  of derived schemes over $S$.
\end{definition}

\begin{definition}
  \label{def:derived_quotient_stack}\cite[Section 3.3]{Toen2014}
  Let $G$ be a smooth group scheme over a base derived scheme $S$
  acting on a derived scheme $X$ over $S$. The \bfem{quotient stack}
  $[X/G]$ is the derived Artin stack $|B_\bullet(X, G)|$,
  defined as the colimit of the simplicial object $B_\bullet(X, G)$:
  \begin{equation*}
    \big[{X}/G\big]= \mathrm{colim}
    \bigg(X\mathrel{\substack{\textstyle\leftarrow\\[-0.1ex]
      \textstyle\leftarrow \\[-0.1ex]}} X \times_S G
      \mathrel{\substack{\textstyle\leftarrow\\[-0.1ex]
          \textstyle\leftarrow \\[-0.1ex]
      \textstyle\leftarrow\\[-0.3ex]}} X \times_S G^2
      \mathrel{\substack{\textstyle\leftarrow\\[-0.1ex]
          \textstyle\leftarrow \\[-0.1ex]
          \textstyle\leftarrow \\[-0.1ex]
      \textstyle\leftarrow \\[-0.3ex]}}
    \cdot\cdot\cdot \bigg),
  \end{equation*}  where the maps are given by the action and projection.
\end{definition}
\begin{definition}
  \label{def:classifying_stack}
  Let $G$ be a smooth group scheme over a base derived scheme $S$.
  The \bfem{classifying stack} $BG$ is the quotient stack $[S/G]$
  defined by the trivial action of $G$ on $S$, \cite[Section 3.3]{Toen2014}.
\end{definition}

The stack quotient integrates with the theory of perfect complexes,
replacing the classical requirement of global slices.

\begin{observation}
  \label{prop:stack_necessity}
  Let $\mathbb{K}$ be a field and let $S =
  \mathrm{Spec}(\mathbb{K})$. Let $\mathbb{G}_m =
  \mathrm{Spec}(\mathbb{K}[t, t^{-1}])$ be the multiplicative group
  scheme over $S$. In the orbit space $\mathbb{A}^1_S / \mathbb{G}_m$
  in the category of ringed spaces, the origin lies in the closure of
  the open orbit. The quotient stack $[\mathbb{A}^1_S /
  \mathbb{G}_m]$ restricts to the classifying stack $B\mathbb{G}_m$
  at the origin and to the scheme $S$ on the open orbit.
\end{observation}
\begin{proof}
  Let the smooth group scheme $\mathbb{G}_m$ over $S$ act on the
  affine line $\mathbb{A}^1_S$ by multiplication. The action admits
  two orbits. The first orbit is the origin, which we identify with
  $S$. The second orbit is the open complement $U = \mathbb{A}^1_S
  \setminus S$. The $\mathbb{G}_m$-invariant open sets in
  $\mathbb{A}^1_S$ are the empty set, $U$, and $\mathbb{A}^1_S$.
  These invariant sets determine the quotient topology on the orbit
  space. The only open set containing the origin is the entire space
  $\mathbb{A}^1_S$, which necessarily contains the open orbit $U$.
  The origin and the open orbit cannot be separated by disjoint open
  sets. In particular, the origin lies in the closure of the open orbit.

  The quotient stack $[\mathbb{A}^1_S / \mathbb{G}_m]$ is the derived
  Artin stack $|B_\bullet(\mathbb{A}^1_S, \mathbb{G}_m)|$. We
  restrict this quotient stack to the invariant subschemes.

  Over the open orbit $U$, the action is simply transitive. The
  restricted quotient groupoid $B_\bullet(U, \mathbb{G}_m)$ is an
  equivalence relation. The quotient stack $|B_\bullet(U,
  \mathbb{G}_m)|$ is the scheme-theoretic quotient $S$.

  Over the origin, the inclusion $S \hookrightarrow \mathbb{A}^1_S$
  is $\mathbb{G}_m$-equivariant with respect to the trivial action on
  $S$. The restriction of the quotient groupoid
  $B_\bullet(\mathbb{A}^1_S, \mathbb{G}_m)$ to the origin yields the
  smooth groupoid of derived schemes $B_\bullet(S, \mathbb{G}_m)$.
  The quotient stack $|B_\bullet(S, \mathbb{G}_m)|$ is the
  classifying stack $B\mathbb{G}_m$.

\end{proof}

\begin{definition}
  \label{def:equivariant_sheaves}
  Let $G$ be a smooth group scheme over a base derived scheme $S$
  acting on a derived scheme $X$ over $S$. The stable
  $\infty$-category of \bfem{$G$-equivariant quasi-coherent sheaves}
  on $X$, denoted $L_{qcoh}^G(X)$, is defined as the limit of the
  cosimplicial $\infty$-category $L_{qcoh}(B_\bullet(X, G))$. This
  cosimplicial diagram is obtained by applying the contravariant
  functor $L_{qcoh}(-)$ to the quotient groupoid $B_\bullet(X, G)$.
\end{definition}
That is, $L_{qcoh}^G(X)$ is defined as the limit of the cosimplicial
object $L_{qcoh}(B_\bullet(X, G))$
\begin{equation*}
  L_{qcoh}^G(X) = \lim
  \bigg(L_{qcoh}(X)\mathrel{\substack{\textstyle\rightarrow\\[-0.1ex]
    \textstyle\rightarrow \\[-0.1ex]}} L_{qcoh}(X \times_S G)
    \mathrel{\substack{\textstyle\rightarrow\\[-0.1ex]
        \textstyle\rightarrow \\[-0.1ex]
    \textstyle\rightarrow\\[-0.3ex]}} L_{qcoh}(X \times_S G^2)
    \mathrel{\substack{\textstyle\rightarrow\\[-0.1ex]
        \textstyle\rightarrow \\[-0.1ex]
        \textstyle\rightarrow \\[-0.1ex]
    \textstyle\rightarrow \\[-0.3ex]}}
  \cdot\cdot\cdot \bigg)
\end{equation*}
where the maps are the pullbacks induced by the action and projection morphisms.

\begin{definition}
  \label{def:principal_torsor}
  Let $G$ be a derived group scheme over a base derived scheme $S$. A
  morphism of derived stacks $p \colon \widetilde{X} \to X$ equipped
  with a $G$-action on $\widetilde{X}$ over $X$ is a \bfem{principal
  $G$-torsor} if the natural morphism $[\widetilde{X}/G] \to X$ is an
  isomorphism in the homotopy category of derived stacks and the square
  \begin{equation*}
    \begin{tikzcd}
      \widetilde{X} \arrow[r] \arrow[d, "p"'] & S \arrow[d] \\
      X \arrow[r] & BG
    \end{tikzcd}
  \end{equation*}
  is homotopy cartesian, where $S \to BG$ is the canonical morphism
  from the base scheme to the classifying stack \cite[Definition
  1.3.5.2]{ToenHAG}.
\end{definition}

Note that the homotopy cartesian square condition is equivalent to
the requirement that the shear map $G \times_S \widetilde{X} \to
\widetilde{X} \times_X \widetilde{X}$ defined by $(g, x) \mapsto (gx,
x)$ is an equivalence.

When $X\in \mathbf{dAff}_\K$,  say $X=\Spec A$ with $A$ a cdga,  a
\bfem{$G$-torsor over $A$} is defined to be a $G$-module $F \in
G\text{-Mod}$, together with a fibration of $G$-modules
\(
  F \longrightarrow \Spec A,
\)
such that there exists an étale covering $A \longrightarrow B$ with
the property that the object
\[
  F \times_{\Spec A} \Spec B \longrightarrow \Spec B
\]
is equivalent over $\Spec B$ to
\(
  G \times \Spec B \longrightarrow \Spec B
\) (where $G$ acts on itself by left translations). For details, we
refer to \cite[Section 3.4]{TV2}.

\subsection{Derived Symplectic, Lagrangian, and Contact Structures}
\label{sec:structures}

\begin{definition}[Derived de Rham Complex and Closed Forms]
  For $A \in \mathrm{cdga}_\K^{\leq 0}$, the \bfem{derived de Rham
  algebra} $\DR(A/\K)$ is the graded mixed complex evaluated as
  $\mathrm{Sym}_A(\LL_{A/\K}[1])$ incorporating the standard Hodge
  filtration \cite[Section 1.4.1]{CPTVV}.
  The \bfem{space of (closed) $p$-forms of degree $n$} on an affine derived
  scheme $\Spec A$ is defined as
  \begin{align*}
    \mathcal{A}^{p}_\K(A,n)
    &:= \mathrm{Map}_{\mathrm{dg}_k}(\K, (\wedge^p\LL_{A/\K})[n]), \\
    \mathcal{A}^{p,cl}_\K(A,n)
    &:= \mathrm{Map}_{\epsilon-\mathrm{dg}_k}(\K(p)[-n-p],\DR(A/\K))
  \end{align*}
  in the $\infty$-category of mixed graded $\K$-modules \cite[Definition
  1.8]{PTVV}. For a derived Artin stack $F$, these are globalized via
  mapping spaces, see \cite[Definition 1.12]{PTVV}, such that
  \begin{equation*}
    \mathcal{A}^{p}_\K(F,n) :=
    \mathrm{Map}_{\mathbf{dSt}_\K}(F, \mathcal{A}^{p}_\K(-, n)),
    \quad \mathcal{A}^{p,cl}_\K(F,n) :=
    \mathrm{Map}_{\mathbf{dSt}_\K}(F, \mathcal{A}^{p,cl}_\K(-, n)).
  \end{equation*}
\end{definition}
More specifically, $\omega \in \mathcal{A}^p(F,n)$ is a $d$-closed
element of $DR(F)$ of weight $p$ and cohomological degree $n$.
Similarly, $\omega \in \mathcal{A}^{p,cl}(F,n)$ is a
$(d+d_{dR})$-closed element  of $DR(F)$ given as a formal series
$\omega=\omega_p+\omega_{p+1} + \cdots$ such that $\omega_q$ has
weight $q$ and degree $n-q+p$. For details, we refer to
\cite[Section~1.2]{PTVV}.

Note also that the space $\mathcal{A}^{p,cl}(F,n)$  is usually more
complicated than the space $ \mathcal{A}^p(F,n)$  even if we make
nice assumptions on $F$. Thanks to \cite[Prop.  1.14]{PTVV}, we have
the following identification for the space $ \mathcal{A}^p(F,n)$,
making it more tractable:
\begin{equation*}
  \mathcal{A}^p(F,n)\simeq {\rm Map}_{L_{qcoh}(F)}(\mathcal{O}_F,
  \wedge^p\mathbb{L}_F[n]).
\end{equation*}
Thus, any 2-form of degree $n$  induces a morphism $\mathcal{O}_F
\rightarrow  \wedge^2\mathbb{L}_F[n]$ in $D_{qcoh}(F)$, and hence, by
duality, a morphism $\mathbb{T}_F \wedge \mathbb{T}_F \rightarrow
\mathcal{O}_F[n]$. By adjunction, this gives the induced morphism
$\Theta_{\omega} : \mathbb{T}_F \rightarrow \mathbb{L}_F [n]$,
leading to the following definition:

\begin{definition}[Shifted Symplectic Structures]
  Following Pantev et al. \cite[Definition 1.18]{PTVV}, an
  \bfem{$n$-shifted symplectic structure} on a derived Artin stack
  $F$ is a point $\omega$ in the space $\mathbb{S}ymp(F, n)$. The
  space $\mathbb{S}ymp(F, n)$ is defined as the homotopy pullback
  \begin{equation*}
    \mathbb{S}ymp(F,n) := \mathcal{A}^{2,cl}(F,n)
    \times_{\mathcal{A}^2(F,n)}^h \mathcal{A}^2(F,n)^{nd}
  \end{equation*}
  where $\mathcal{A}^2(F,n)^{nd}$ is the subspace of non-degenerate
  2-forms, i.e., those whose underlying morphism $\Theta_{\omega}:
  \TT_F \to \LL_F[n]$ is an isomorphism in $D_{qcoh}(F)$.
\end{definition}

The shift $n$ compensates for the homological
degrees inherent in the tangent and cotangent complexes. This shift
distinguishes derived symplectic geometry from its classical counterpart.
Let us mention here two important existence results in the context of
\cite{PTVV}.

\begin{theorem}[Calaque's Shifted Cotangent Stack Theorem] \label{thm:calaque}
  Let $X$ be a derived Artin stack. The $n$-shifted cotangent stack,
  defined via the relative derived spectrum as
  \begin{align*}
    T^*[n]X
    \simeq\mathbb{R}\mathbf{Spec}_X(\mathbb{L}\mathrm{Sym}_{\cO_X}(\TT_X[-n])),
  \end{align*}
  carries a canonical $n$-shifted symplectic structure
  \cite[Theorem~2.4]{Calaque2019}.
\end{theorem}

As for classical cotangent bundles, the symplectic structure on a
shifted cotangent stack is the de Rham differential of its
tautological $1$-form \cite[Section 2.1]{Calaque2019}.

\begin{theorem}[PTVV Mapping Stack Theorem] \label{thm:ptvv}
  Let $F$ be a derived Artin stack equipped with an $n$-shifted
  symplectic structure, and let $Y$ be an $\cO$-compact derived stack
  equipped with an $\cO$-orientation of dimension $d$. Then the
  derived mapping stack  $\mathbb{R}\mathrm{Map}(Y, F)$
  has    a canonical $(n-d)$-shifted symplectic structure
  \cite[Theorem 2.5]{PTVV}.
\end{theorem} Here ``orientations," see \cite[Definition 2.4]{PTVV},
on derived stacks provide the homological equivalent of integration.
This condition is a
prerequisite for pushing forward symplectic structures along mapping
spaces, analogous to integration along the fiber.
Furthermore, the mapping stack theorem generates examples of
shifted symplectic stacks. Moduli spaces of sheaves or
representations inherit shifted
symplectic structures from their domain curves or surfaces via the
mapping stack theorem.

Let us give some applications of the mapping stack theorem above.
\begin{example}
  Let $ G $ be an affine algebraic
  group  equipped with a nondegenerate invariant symmetric bilinear
  pairing  on its Lie algebra. Then the classifying stack is given as
  the functor   \[
    \mathsf{B G} : Z \in \mathbf{dAff}_\K^{\mathrm{op}} \longmapsto
    \{\text{principal } G\text{-bundle on } Z\}
    \in \mathcal{S}.
  \]Recall that $\mathsf{B G}$ has a \(2\)-shifted symplectic
  structure \cite[Section 1.2, Classifying stacks]{PTVV}.
  Furthermore, there is an equivalence $\mathbb{T}_{\mathsf{B
  G}}\simeq \mathfrak{g}[1]$, where $\mathfrak{g}$ is the Lie algebra
  of G with the adjoint representation.

  Likewise, the \textit{moduli of perfect complexes} can be defined
  as the derived stack \[
    \mathsf{Perf} : Z \in \mathbf{dAff}_\K^{\mathrm{op}} \longmapsto
    \{\text{perfect complexes on }   Z\}
    \in \mathcal{S}
  \] such that it also carries a canonical \(2\)-shifted symplectic
  structure \cite[Theorem 2.12]{PTVV} and that its tangent complex
  $\mathbb{T_{\mathsf{Perf}}}$ is equivalent to $E^{\vee} \otimes
  E[1]$, with $E$ the universal complex.

  Combining these examples with the PTVV's mapping stack theorem
  above, let $Y$ be a smooth and projective Calabi-Yau $d$-fold, then
  by \cite[Theorem~2.5 and Corollary~2.13(3)]{PTVV}:
  \begin{enumerate}
    \item The derived Artin stack
      \( \mathrm{Perf}(Y) :=  \R\mathrm{Map}(Y, \mathrm{Perf}) \)
      of perfect complexes on $Y$ is \((2-d)\)-shifted symplectic.
    \item For a reductive group $G$, the derived Artin stack
      \( \mathrm{Bun}_G(Y) := \R\mathrm{Map}(Y, \mathsf{B G}) \)
      of principal $G$-bundles on $Y$ is \((2-d)\)-shifted symplectic.
  \end{enumerate}
\end{example}

\begin{example}
  Fix the pair $(\mathsf{B G}, \omega)$ with $\omega \in
  \mathbb{S}ymp(\mathsf{B G}, 2)$. Given  a smooth and proper curve $C$, we call
  \begin{equation*}
    \mathrm{LocSys}_G(C) := \R\mathrm{Map} (C_{dR},\mathrm{BG})
  \end{equation*}
  the \bfem{moduli stack of flat $ G $-connections on $ C $}. Then
  the following hold.
  \begin{enumerate}
    \item $ \mathrm{LocSys}_G(C)$ is 0-shifted symplectic by Theorem
      \ref{thm:ptvv}.
    \item There is a natural closed 1-form of degree 1 on $
      \mathrm{Bun}_G(C) $, denoted by $\int_C ev^* \omega$, such that
      we have the identification \cite[Prop. 1.24]{Safronov2023}
      \begin{equation*}
        \mathrm{LocSys}_G(C)\simeq {\rm T}_{\int_C ev^*
        \omega}^*\mathrm{Bun}_G(C).
      \end{equation*}
  \end{enumerate} That is, $ \mathrm{LocSys}_G(C) $ can be
  equivalently seen as the \textit{twisted cotangent stack} of $
  \mathrm{Bun}_G(C)$ with the twisting 1-form $ \int_C ev^* \omega
  \in \mathcal{A}^{1,cl}(\mathrm{Bun}_G(C), 1).$
\end{example}

Let us now list other key concepts, along with certain remarks and observations.
\begin{definition}[Isotropic and Lagrangian Structures]
  Let $(X, \omega)$ be an $n$-shifted symplectic derived stack. A map
  $f: L \to X$ is equipped with an \bfem{isotropic structure} if
  there is a specified path $h$ between $0$ and $f^*\omega$ in the
  Kan complex of
  closed 2-forms $\mathcal{A}^{2, cl}(L, n)$ \cite[Definition
  2.7]{PTVV}.

  The Kan complex of isotropic structures is
  defined as $$Isot(f, \omega) := Path_{0,
  f^*\omega}(\mathcal{A}^{2,cl}(L, n)).$$ This
  null-homotopy induces a morphism $\Theta_h: \TT_L \to
  \LL_{L/X}[n-1]$. An isotropic structure is
  \bfem{Lagrangian} if $\Theta_h$ is an equivalence in the stable
  $\infty$-category $L_{qcoh}(L)$ \cite[Definition 2.8]{PTVV}.
\end{definition}

The requirement of a specified null-homotopy replaces the classical
condition of a vanishing differential form. In an $\infty$-category,
the space of null-homotopies carries higher categorical structure.
The choice of isotropic structure is part of the geometric
data.

\begin{definition}[Shifted Contact Structures] \label{berktav:contact}
  Following Berktav \cite[Section 3.2]{kib1}, \cite[Definition
  2.17]{kib2}, an \bfem{$n$-shifted contact structure} on a derived
  Artin stack $X$ consists of a line bundle $\cL$ on $X$ that is
  trivial on an \'etale cover of $X$, an $n$-shifted $1$-form $\alpha
  \in \pi_0 \mathbb{R}\Gamma(X, \LL_X \otimes^{\mathbb{L}}_{\cO_X}
  \cL[n])$, and a fiber sequence of perfect complexes built from the
  dual contraction map $\alpha^\vee \colon \TT_X \to \cL[n]$
  \begin{equation}
    \cK \to \TT_X \xrightarrow{\alpha^\vee} \cL[n]
  \end{equation}
  such that the underlying 2-form of the derived de Rham differential
  $d_{dR}\alpha$ induces an equivalence $\cK \xrightarrow{\sim}
  \cK^\vee \otimes^{\mathbb{L}}_{\cO_X} \cL[n]$ in $L_{qcoh}(X)$,
  rendering $\cK$ a symplectic complex. Here $d_{dR}\alpha$ is
  computed on an \'etale cover on which $\cL$ is trivial, and its
  restriction to $\cK$ does not depend on the trivialization, since
  for an invertible function $u$ we have $d_{dR}(u\alpha) = d_{dR}u
  \wedge \alpha + u\, d_{dR}\alpha$ and $\alpha$ vanishes on $\cK$.
\end{definition}

Shifted contact structures provide the
odd-dimensional counterpart to shifted symplectic
stacks. The line bundle accounts for the scaling ambiguity in contact
geometry, allowing the structure to be defined globally
across charts.

\begin{definition}[Derived Legendrian
  Structures]\label{definition:legendrian structures}
  Let $(X, \cL, \cK, \alpha)$ be an $n$-shifted contact derived
  stack. A map $f \colon L \to X$ is \bfem{isotropic} if it is
  equipped with a null-homotopy $f^*\alpha \simeq 0$.

  An isotropic structure is called \bfem{Legendrian} if it induces a
  stable fiber sequence in $L_{qcoh}(L)$ of the form
  \begin{equation*}
    \cO_L \to \LL_{L/X} \otimes^{\mathbb{L}}_{\cO_L} f^*\cL[n-1] \to \TT_L.
  \end{equation*}
\end{definition}

\begin{remark}\label{rmk:legendrian_sequence}
  The \bfem{Legendrian} condition above is equivalent to the one
  given in \cite[Observation~3.14 and Definition~3.15]{kib3}. In
  \cite[Observation~3.14]{kib3}, the isotropic lift $\TT_L \to
  f^*\cK$ induces a stable fiber sequence
  \begin{equation*}
    \TT_L \to f^*\cK \to \LL_L \otimes^{\mathbb{L}}_{\cO_L} f^*\cL[n].
  \end{equation*}
  Let $C := \mathrm{cofib}(\TT_L \to f^*\cK)$. The condition is $C
  \simeq \LL_L \otimes^{\mathbb{L}}_{\cO_L} f^*\cL[n]$.

  By definition of $\cK$, there is a sequence $f^*\cK \to f^*\TT_X
  \to f^*\cL[n]$. The relative tangent sequence for $f$ is $\TT_{L/X}
  \to \TT_L \to f^*\TT_X$.
  The pasting law for exact triangles of fibers (the cofiber version
  of Lemma \ref{lem:pasting_law}) applied to $\TT_L \to f^*\cK \to
  f^*\TT_X$ produces a stable fiber sequence of cofibers
  \begin{equation*}
    C \to \TT_{L/X}[1] \to f^*\cL[n].
  \end{equation*}
  A shift by $-1$ and substitution of $C$ gives
  \begin{equation*}
    \LL_L \otimes^{\mathbb{L}}_{\cO_L} f^*\cL[n-1] \to \TT_{L/X} \to
    f^*\cL[n-1].
  \end{equation*}
  Taking the $\cO_L$-linear dual and tensoring with $f^*\cL[n-1]$ gives
  \begin{equation*}
    \cO_L \to \LL_{L/X} \otimes^{\mathbb{L}}_{\cO_L} f^*\cL[n-1] \to \TT_L.
  \end{equation*}
\end{remark}

The Legendrian condition establishes the derived geometric analogue
for maximally isotropic submanifolds. As in the classical
theory, derived Legendrian structures are governed by the
underlying contact distribution, providing the geometric foundation
for studying boundary conditions in topological field theories.
\section{Transversality}
\subsection{Classical Transversality} \label{sec:classical}

\begin{definition}[Liouville Vector Field]
  Let $(W, \omega)$ be a symplectic manifold. A vector field $Y$ on
  $W$ is called a \bfem{Liouville vector field} if its Lie
  derivative satisfies $\mathcal{L}_Y \omega = \omega$.
\end{definition}

\begin{lemma}[{\cite[Lemma 1.4.5]{Geiges}}] \label{lem:classical_geiges}
  Let $(W, \omega)$ be a symplectic manifold of dimension $2n$, and
  let $Y$ be a Liouville vector field on $W$. If $M \subset W$ is a
  smooth hypersurface transverse to $Y$, then the restriction $\alpha
  := (\iota_Y \omega)|_M$ is a contact form on $M$.
\end{lemma}
\begin{proof}
  By Cartan's magic formula and the assumption that the symplectic
  form is closed ($d\omega = 0$), the following holds
  \begin{equation*}
    d(\iota_Y \omega) = \mathcal{L}_Y \omega - \iota_Y d\omega = \omega.
  \end{equation*}
  Therefore, $d\alpha = \omega|_M$. Showing that $\alpha$ is a
  contact form requires verifying that $\alpha \wedge (d\alpha)^{n-1}$
  is a nowhere vanishing volume form on $M$.

  On the symplectic manifold $W$, the top-dimensional form $\omega^n$
  is a volume form. Contracting this with the Liouville vector field
  $Y$ we obtain
  \begin{equation*}
    \iota_Y(\omega^n) = n (\iota_Y \omega) \wedge \omega^{n-1}.
  \end{equation*}
  Because $Y$ is transverse to the hypersurface $M$, the tangent
  space admits the pointwise coordinate splitting $T_x W \cong T_x M
  \oplus \mathbb{R}Y_x$.

  Evaluating the top-form $\omega^n$ on the frame $(Y_x, e_1, \dots,
  e_{2n-1})$, where the vectors $e_i$ span the local tangent space
  $T_x M$, shows that the transverse contraction
  $\iota_Y(\omega^n)|_M$ defines a nowhere vanishing volume form on
  $M$. Restricting the identity to $M$ preserves this
  property, yielding
  \begin{equation*}
    n \alpha \wedge (d\alpha)^{n-1} \neq 0,
  \end{equation*}
  proving $\alpha$
  is a contact form on $M$.
\end{proof}

The classical proof relies on point-set transversality
(the splitting of the tangent space) and the pointwise evaluation
of differential forms on topological frames. In derived moduli
theory, intersections are rarely transverse and tangent
spaces jump in rank, requiring a homological replacement for
this coordinate argument.
\subsection{The Derived Setup and Statement} \label{sec:setup}

The classical Liouville vector field is replaced by a group action in
the derived setting.

\begin{definition}[Lie Algebra of a Group Stack]
  For an affine smooth group scheme $G$ with Lie algebra
  $\mathfrak{g}$, the cotangent complex of the classifying stack $BG$
  is $\mathfrak{g}^\vee[-1]$. A $G$-action on a stack $\widetilde{X}$
  induces an infinitesimal Lie algebra action morphism
  $\mathfrak{g} \otimes_\K \cO_{\widetilde{X}} \to
  \TT_{\widetilde{X}}$ in $L_{qcoh}(\widetilde{X})$ \cite[Section 1.2]{PTVV}.
\end{definition}

\begin{lemma}[Derived Atiyah Sequence] \label{lem:atiyah}
  Let $G$ be a smooth affine group scheme with Lie algebra
  $\mathfrak{g}$. For a principal $G$-torsor $p \colon \widetilde{X}
  \to X$ in $\dArt_\K$, the $G$-action induces a stable fiber
  sequence in $L_{qcoh}(\widetilde{X})$
  \begin{equation*}
    \mathbb{L}p^*\LL_{X} \to \LL_{\widetilde{X}} \to
    \mathfrak{g}^\vee \otimes_\K \cO_{\widetilde{X}}.
  \end{equation*}
\end{lemma}
\begin{proof}
  By \cite[Definition 1.4.1.15]{ToenHAG}, the morphism $p \colon
  \widetilde{X} \to X$ has a relative cotangent complex. By
  \cite[Lemma 1.4.1.16(1)]{ToenHAG}, for any affine point $x \colon
  \Spec(A) \to \widetilde{X}$, there is a homotopy cofiber sequence
  of stable $A$-modules
  \begin{equation} \label{eq:local_transitivity}
    \LL_{X, p \circ x} \to \LL_{\widetilde{X}, x} \to \LL_{\widetilde{X}/X, x}.
  \end{equation}

  By \cite[Definitions 1.4.1.7 and 1.4.1.15]{ToenHAG}, the pointwise
  complexes are the pullbacks of the global cotangent complexes,
  implying $$\LL_{\widetilde{X}, x} \simeq x^* \LL_{\widetilde{X}},
  \quad \LL_{\widetilde{X}/X, x} \simeq x^* \LL_{\widetilde{X}/X}, \;
  \text{ and } \ \LL_{X, p \circ x} \simeq x^* \mathbb{L}p^* \LL_X.$$

  Substituting these identifications into
  \eqref{eq:local_transitivity} gives a stable fiber sequence in
  $L_{qcoh}(\Spec(A))$ for every affine point $x$
  \begin{equation*}
    x^* \mathbb{L}p^* \LL_X \to x^* \LL_{\widetilde{X}} \to x^*
    \LL_{\widetilde{X}/X}.
  \end{equation*}

  The stable $\infty$-category $L_{qcoh}(\widetilde{X})$ is
  equivalent to the homotopy limit of $L_{qcoh}(\Spec(A))$ over all
  affine points mapping to $\widetilde{X}$. Pointwise exactness
  establishes the global transitivity fiber sequence in
  $L_{qcoh}(\widetilde{X})$
  \begin{equation} \label{eq:standard_transitivity}
    \mathbb{L}p^*\LL_{X} \to \LL_{\widetilde{X}} \to \LL_{\widetilde{X}/X}.
  \end{equation}

  Because $p$ is a principal $G$-torsor, the action map $\rho \colon
  G \times \widetilde{X} \to \widetilde{X}$ and the projection
  $\mathrm{pr}_2 \colon G \times \widetilde{X} \to \widetilde{X}$
  define a homotopy pullback square in $\dArt_\K$
  \[
    \begin{tikzcd}
      G \times \widetilde{X} \arrow[r, "\rho"] \arrow[d,
      "\mathrm{pr}_2"'] & \widetilde{X} \arrow[d, "p"] \\
      \widetilde{X} \arrow[r, "p"'] & X
    \end{tikzcd}
  \]

  By \cite[Lemma 1.4.1.16(2)]{ToenHAG}, base change along $p$ gives
  the equivalence
  \begin{equation*}
    \mathbb{L}\rho^* \LL_{\widetilde{X}/X} \simeq \LL_{G \times
    \widetilde{X} / \widetilde{X}}.
  \end{equation*}

  The morphism $\mathrm{pr}_2$ is the base change of the structural
  morphism $G \to \Spec(\K)$ along $\pi_{\widetilde{X}} \colon
  \widetilde{X} \to \Spec(\K)$. Applying \cite[Lemma
  1.4.1.16(2)]{ToenHAG} gives the equivalence $\LL_{G \times
  \widetilde{X} / \widetilde{X}} \simeq \mathbb{L}\mathrm{pr}_1^*
  \LL_{G}$, yielding
  \begin{equation*}
    \mathbb{L}\rho^* \LL_{\widetilde{X}/X} \simeq
    \mathbb{L}\mathrm{pr}_1^* \LL_{G}.
  \end{equation*}

  Pulling back this equivalence along the section $s = (e,
  \mathrm{id}_{\widetilde{X}}) \colon \widetilde{X} \to G \times
  \widetilde{X}$ gives
  \begin{equation*}
    \mathbb{L}s^* \mathbb{L}\rho^* \LL_{\widetilde{X}/X} \simeq
    \mathbb{L}s^* \mathbb{L}\mathrm{pr}_1^* \LL_{G}.
  \end{equation*}

  The section $s$ satisfies the composition identities $\rho \circ s
  = \mathrm{id}_{\widetilde{X}}$ and $\mathrm{pr}_1 \circ s = e \circ
  \pi_{\widetilde{X}}$, shown in the commutative diagrams
  \[
    \begin{tikzcd}
      \widetilde{X} \arrow[r, "s"] \arrow[dr,
      "\mathrm{id}_{\widetilde{X}}"'] & G \times \widetilde{X}
      \arrow[d, "\rho"] \\
      & \widetilde{X}
    \end{tikzcd}
    \quad \quad
    \begin{tikzcd}
      \widetilde{X} \arrow[r, "s"] \arrow[d, "\pi_{\widetilde{X}}"']
      & G \times \widetilde{X} \arrow[d, "\mathrm{pr}_1"] \\
      \Spec(\K) \arrow[r, "e"'] & G
    \end{tikzcd}
  \]

  These identities reduce the pullback equivalence to
  \begin{equation*}
    \LL_{\widetilde{X}/X} \simeq \mathbb{L}\pi_{\widetilde{X}}^*
    \mathbb{L}e^* \LL_{G}.
  \end{equation*}

  For a smooth affine group scheme $G$, the pullback of the cotangent
  complex along the identity section gives $\mathbb{L}e^* \LL_{G}
  \simeq \mathfrak{g}^\vee$. This establishes the equivalence
  $\LL_{\widetilde{X}/X} \simeq \mathfrak{g}^\vee
  \otimes_{\mathbb{K}} \cO_{\widetilde{X}}$. Substituting this
  identification into \eqref{eq:standard_transitivity} produces the
  stable fiber sequence $$\mathbb{L}p^*\LL_{X} \to
  \LL_{\widetilde{X}} \to \mathfrak{g}^\vee \otimes_{\mathbb{K}}
  \cO_{\widetilde{X}}.$$
\end{proof}

\begin{corollary} \label{cor:fundamental_morphism}
  Let $p \colon \widetilde{X} \to X$ be a principal $\Gm$-torsor in
  $\dArt_\K$. The standard left-invariant vector field on $\Gm$
  determines a fundamental morphism $Y \colon \cO_{\widetilde{X}} \to
  \TT_{\widetilde{X}}$ in $L_{qcoh}(\widetilde{X})$ such that the
  relative cotangent sequence takes the form
  \begin{equation*}
    \LL p^* \LL_X \xrightarrow{(\TT p)^\vee} \LL_{\widetilde{X}}
    \xrightarrow{Y^\vee} \cO_{\widetilde{X}}.
  \end{equation*}
\end{corollary}
\begin{proof}
  The proof of Lemma \ref{lem:atiyah} establishes the equivalence
  $\LL_{\widetilde{X}/X} \simeq \mathfrak{g}^\vee \otimes_\K
  \cO_{\widetilde{X}}$. Taking the derived
  $\cO_{\widetilde{X}}$-linear dual of this equivalence gives an isomorphism
  \begin{equation*}
    \TT_{\widetilde{X}/X} \simeq \mathfrak{g} \otimes_\K \cO_{\widetilde{X}}.
  \end{equation*}
  The relative tangent complex sits in the transitivity fiber
  sequence $\TT_{\widetilde{X}/X} \to \TT_{\widetilde{X}}
  \xrightarrow{\TT p} \LL p^* \TT_X$. Composing the isomorphism with
  the canonical morphism $\TT_{\widetilde{X}/X} \to
  \TT_{\widetilde{X}}$ defines the infinitesimal action morphism
  \begin{equation*}
    a \colon \mathfrak{g} \otimes_\K \cO_{\widetilde{X}} \to
    \TT_{\widetilde{X}}.
  \end{equation*}
  The multiplicative group scheme $\Gm$ has the coordinate algebra
  $\K[t, t^{-1}]$. The tangent space at the identity is the
  $1$-dimensional $\K$-vector space generated by the derivation
  $\partial = t \partial_t$. This generator defines an isomorphism of
  $\K$-vector spaces $\K \xrightarrow{\sim} \mathfrak{g}$. The
  fundamental morphism $Y$ is defined as the composition
  \begin{equation*}
    Y \colon \cO_{\widetilde{X}} \simeq \K \otimes_\K
    \cO_{\widetilde{X}} \xrightarrow{\sim} \mathfrak{g} \otimes_\K
    \cO_{\widetilde{X}} \xrightarrow{a} \TT_{\widetilde{X}}.
  \end{equation*}
  Taking the derived dual of this composition we obtain
  \begin{equation*}
    Y^\vee \colon \LL_{\widetilde{X}} \xrightarrow{a^\vee}
    \mathfrak{g}^\vee \otimes_\K \cO_{\widetilde{X}}
    \xrightarrow{\sim} \cO_{\widetilde{X}}.
  \end{equation*}
  The morphism $a^\vee$ is the canonical projection
  $\LL_{\widetilde{X}} \to \LL_{\widetilde{X}/X}$ in the cotangent
  transitivity sequence, post-composed with the equivalence from
  Lemma \ref{lem:atiyah}. Substituting this identification into the
  standard transitivity sequence gives the stable fiber sequence
  \begin{equation*}
    \LL p^* \LL_X \xrightarrow{(\TT p)^\vee} \LL_{\widetilde{X}}
    \xrightarrow{Y^\vee} \cO_{\widetilde{X}}.
  \end{equation*}
\end{proof}

The $\Gm$-action allows the formulation of the contact space
without local analytic arguments \cite[Section~4]{kib1}.

\begin{definition}[Derived Liouville Condition]
  The $\Gm$-action has \textbf{weight 1} with respect to
  $\omega_{\widetilde{X}}$ if $\omega_{\widetilde{X}}$ is a weight 1 eigenform.
  Equivalently, the derived pullback of the action map
  $\rho: \Gm \times \widetilde{X} \to \widetilde{X}$ satisfies $\rho^*
  \omega_{\widetilde{X}} \simeq z \cdot \pi_{\widetilde{X}}^*
  \omega_{\widetilde{X}}$,
  where $z$ is the character coordinate on $\Gm$ \cite[Section
  2.1.1]{Calaque2019}. Here both sides are regarded as forms on $\Gm
  \times \widetilde{X}$ relative to the first factor $\Gm$. As
  absolute forms the two sides differ. For instance, for the action
  of $\Gm$ on itself we have $\rho(z,t) = zt$, so $\rho^*(d_{dR}t) =
  z\, d_{dR}t + t\, d_{dR}z$.
\end{definition}

In classical geometry, the Liouville vector field dilates the
symplectic form. The weight 1 condition is the algebraic
translation of this dilation, rendering the symplectic form
homogeneous of degree 1. This homogeneity ensures the contact form
survives the quotient process.
The derived analogue of the classical transversality lemma can then
be stated as follows.

\begin{theorem}[Derived Analogue of Classical Transversality]
  \label{thm:derived transversality}
  Let $(\widetilde{X}, \omega_{\widetilde{X}})$ be an $n$-shifted symplectic
  derived Artin stack equipped with a $\Gm$-action of weight 1. Let
  $X$ be the stack quotient $X := [\widetilde{X} / \Gm]$, and let $p \colon
  \widetilde{X} \to X$ be the projection. Suppose that $p$ is trivial
  on an \'etale cover of $X$. Then $X$ inherits an
  $n$-shifted contact structure.
\end{theorem}

\begin{remark} \label{rmk:darboux_atlas}
  For $n < 0$, the descended contact structure
  on $X$ locally admits a contact Darboux atlas \cite[Theorem 3.7]{kib2}.
\end{remark}

\begin{remark} \label{rmk:etale}
  The hypothesis on $p$ is needed for
  Definition~\ref{berktav:contact} to apply. It fails near a point
  whose stabilizer acts on the fiber of $\cL$ through a character
  that is non-trivial on the identity component of the stabilizer,
  since the identity component still acts in this way on the pullback
  of $\cL$ to any \'etale neighborhood. Without the hypothesis, we
  call the weight $1$ symplectic structure $\omega_{\widetilde{X}}$
  an \bfem{$n$-shifted contact structure on $X$ in the sense of the
  symplectification}. Likewise, for a morphism $f \colon L \to X$ and
  $\widetilde{L} := L \times_X^h \widetilde{X}$, a \bfem{Legendrian
  structure on $f$ in the sense of the symplectification} is a
  Lagrangian structure on $\widetilde{f} \colon \widetilde{L} \to
  \widetilde{X}$ whose isotropic structure is a path of closed
  $2$-forms of weight $1$. When $p$ is trivial on an \'etale cover,
  Theorem~\ref{thm:derived transversality} turns the first notion
  into the one of Definition~\ref{berktav:contact}.
\end{remark}

Taking the quotient of a homogeneous symplectic space descends the
symplectic data to a contact structure, avoiding a transverse
hypersurface. This reverses the symplectification of
\cite[Theorem~4.7]{kib1} and \cite[Theorem~3.9]{kib2}.

\subsection{Proof of Theorem \ref{thm:A} and Analysis of
Transversality}\label{sec:proof}

\begin{observation}
  \label{prop:transversality_failure}
  Transversality between a global vector field and a closed immersion
  of virtual codimension $1$ fails at every point of the closed
  immersion where the vector field vanishes.
\end{observation}
\begin{proof}
  Let $\widetilde{X}$ be a derived Artin stack. Let $Y \colon
  \cO_{\widetilde{X}} \to \TT_{\widetilde{X}}$ be a global vector
  field in $L_{qcoh}(\widetilde{X})$. Let $i \colon H \to
  \widetilde{X}$ be a closed immersion. The derived normal complex
  $\mathcal{N}_{H/\widetilde{X}}$ is the cofiber in the tangent
  transitivity sequence
  \begin{align*}
    \TT_H \xrightarrow{di} i^*\TT_{\widetilde{X}} \to
    \mathcal{N}_{H/\widetilde{X}}.
  \end{align*}
  Classical transversality requires the composite morphism $c \colon
  \cO_H \xrightarrow{i^*Y} i^*\TT_{\widetilde{X}} \to
  \mathcal{N}_{H/\widetilde{X}}$ to be an equivalence in
  $L_{qcoh}(H)$. This equivalence establishes a homological splitting
  of the tangent sequence, generating a rank-$1$ free summand on $H$.

  The derived zero locus $Z(Y)$ is the homotopy pullback of the
  section $Y$ and the zero section $0 \colon \cO_{\widetilde{X}} \to
  \TT_{\widetilde{X}}$ in the $\infty$-category of derived stacks.

  Suppose there exists an $R$-point $x \colon \mathrm{Spec}(R) \to H
  \times_{\widetilde{X}}^h Z(Y)$ in the derived intersection of the
  hypersurface and the zero locus. By the definition of the homotopy
  pullback $Z(Y)$, the restricted vector field $x^*Y \colon
  \cO_{\mathrm{Spec}(R)} \to x^*\TT_{\widetilde{X}}$ is equipped with
  a null-homotopy $x^*Y \simeq 0$.

  This null-homotopy trivializes the pulled-back composite morphism
  $x^*c$. Since $c$ is an equivalence, the pullback $x^*c$ must be an
  equivalence in $L_{qcoh}(\mathrm{Spec}(R))$. An equivalence between
  rank-$1$ free modules cannot be null-homotopic over a non-trivial
  ring $R$. The existence of zeros on $H$ obstructs the equivalence
  $c$, so global transversality fails.
\end{proof}

\begin{remark}
  \label{rmk:slice_vs_quotient}
  The stack quotient $X \simeq |B_\bullet(\widetilde{X},
  \mathbb{G}_m)|$ circumvents the transversality failure by absorbing
  the zeros of the vector field into its stabilizer structure. The
  vector field $Y$ is the infinitesimal generator of the
  $\mathbb{G}_m$-action on $\widetilde{X}$. The points where $Y$
  vanishes are the points fixed by the $\mathbb{G}_m$-action. Instead
  of producing a topological singularity in the quotient, the fixed
  points acquire $\mathbb{G}_m$ as their stabilizer group in the
  resulting $\infty$-groupoid.

  The projection $p \colon \widetilde{X} \to X$ defines the tangent
  fiber sequence $\cO_{\widetilde{X}} \xrightarrow{Y}
  \TT_{\widetilde{X}} \to \mathbb{L}p^* \TT_X$ in
  $L_{qcoh}(\widetilde{X})$. For an $R$-point $x \colon
  \mathrm{Spec}(R) \to Z(Y)$, the vanishing of the vector field
  corresponds to the null-homotopy $x^*Y \simeq 0$. This
  null-homotopy induces the splitting $$x^*\mathbb{L}p^*\TT_X \simeq
  x^*\TT_{\widetilde{X}} \oplus \cO_{\mathrm{Spec}(R)}[1] \quad
  \text{ in } L_{qcoh}(\mathrm{Spec}(R)).$$ The homological failure
  of transversality is converted into the shifted Lie algebra
  $\cO_{\mathrm{Spec}(R)}[1]$ of the stabilizer group. The complex
  $\mathbb{L}p^*\TT_X$ remains perfect globally because
  $\mathrm{Perf}(\widetilde{X})$ is a stable subcategory of
  $L_{qcoh}(\widetilde{X})$. This mechanism preserves geometric
  dualizability without imposing a global homological splitting on
  the tangent sequence.
\end{remark}

\begin{proof}[\textbf{Proof of Theorem \ref{thm:derived transversality}}]
  The proof proceeds in four steps.

  We first trivialize the relative cotangent complex of the principal
  $\Gm$-torsor $p \colon \widetilde{X} \to X$ using the infinitesimal
  $\Gm$-action. We then construct an $n$-shifted $1$-form
  $\alpha_{\widetilde{X}}$ on $\widetilde{X}$ by contracting the
  symplectic form $\omega_{\widetilde{X}}$ with the fundamental
  vector field of the action. We prove this $1$-form is horizontal
  with respect to $p$. The horizontality and the weight 1 equivariant
  structure of $\omega_{\widetilde{X}}$ allow the descent of
  $\alpha_{\widetilde{X}}$ to a twisted $1$-form $\alpha$ on the
  quotient stack $X$. Finally, we define the derived contact
  distribution as the homotopy fiber of the descended $1$-form and
  deduce its non-degeneracy from the underlying symplectic structure
  on $\widetilde{X}$.
  \medskip

  \bfem{Step 1: The Relative Cotangent Sequence.}
  Let $X$ denote the stack quotient $[\widetilde{X} / \Gm]$. By
  definition, the projection $p \colon \widetilde{X} \to X$ is a
  principal $\Gm$-torsor. By Corollary
  \ref{cor:fundamental_morphism}, the torsor determines a fundamental
  morphism $Y$ and induces the relative cotangent sequence
  \begin{equation*}
    \LL p^* \LL_X \xrightarrow{(\TT p)^\vee} \LL_{\widetilde{X}}
    \xrightarrow{Y^\vee} \cO_{\widetilde{X}}.
  \end{equation*}

  \medskip
  \bfem{Step 2: Construction of the 1-Form.}
  The derived Artin stack $\widetilde{X}$ is equipped with an
  $n$-shifted symplectic structure $\omega_{\widetilde{X}}$. By
  \cite[Definition 1.18]{PTVV}, the
  non-degeneracy of $\omega_{\widetilde{X}}$ gives an equivalence of perfect
  complexes $\Theta_{\omega} \colon \TT_{\widetilde{X}} \xrightarrow{\sim}
  \LL_{\widetilde{X}}[n]$.
  Post-composing the fundamental morphism $Y \colon \cO_{\widetilde{X}} \to
  \TT_{\widetilde{X}}$ with this equivalence yields a morphism in
  $L_{qcoh}(\widetilde{X})$,
  identified via the Yoneda lemma
  $\mathbb{R}\Gamma(\widetilde{X}, E) \simeq
  \mathrm{Map}_{L_{qcoh}(\widetilde{X})}(\cO_{\widetilde{X}}, E)$,
  with an $n$-shifted $1$-form global section
  $\alpha_{\widetilde{X}} \in \pi_0 \mathbb{R}\Gamma(\widetilde{X},
  \LL_{\widetilde{X}}[n])$ \cite[Section~4]{kib1} such that
  \begin{equation*}
    \alpha_{\widetilde{X}} := \Theta_{\omega} \circ Y.
  \end{equation*}
  By the defining equivalence, this global section corresponds
  to a contraction morphism
  $\alpha_{\widetilde{X}}^\vee \colon \TT_{\widetilde{X}} \to
  \cO_{\widetilde{X}}[n]$.

  The composition $\cO_{\widetilde{X}} \xrightarrow{Y}
  \TT_{\widetilde{X}} \xrightarrow{\alpha_{\widetilde{X}}^\vee}
  \cO_{\widetilde{X}}[n]$ corresponds to the derived self-pairing of
  $Y$ with respect to $\omega_{\widetilde{X}}$. In the graded mixed
  derived de Rham complex \cite[Section 1.4.1]{CPTVV}, the graded
  commutativity of the contraction operator implies that the interior
  product $\iota_Y \iota_Y \omega_{\widetilde{X}}$ admits a canonical
  null-homotopy. This defines an equivalence
  $\omega_{\widetilde{X}}(Y, Y) \simeq 0$ within the mapping space
  $\mathrm{Map}_{L_{qcoh}(\widetilde{X})}(\cO_{\widetilde{X}},
  \cO_{\widetilde{X}}[n])$.

  Applying the contravariant derived internal Hom functor
  $\mathbb{R}\underline{\mathcal{H}om}( - , \cO_{\widetilde{X}}[n])$
  to the exact triangle of tangent complexes $\cO_{\widetilde{X}}
  \xrightarrow{Y}
  \TT_{\widetilde{X}} \xrightarrow{\mathbb{T}p} \mathbb{L}p^*\TT_X
  \xrightarrow{+1}$
  gives the dual exact triangle
  \begin{equation*}
    \mathbb{L}p^*\LL_X[n] \xrightarrow{(\mathbb{T}p)^\vee}
    \LL_{\widetilde{X}}[n] \xrightarrow{Y^\vee}
    \cO_{\widetilde{X}}[n] \xrightarrow{+1}
  \end{equation*}
  in $L_{qcoh}(\widetilde{X})$.
  Evaluating global sections by applying the functor
  $\mathrm{Map}_{L_{qcoh}(\widetilde{X})}(\cO_{\widetilde{X}}, -)$ to
  this exact triangle generates a long exact sequence in cohomology
  containing the segment
  \begin{equation*}
    \cdots \to \pi_0 \mathrm{Map}(\cO_{\widetilde{X}},
    \mathbb{L}p^*\LL_X[n]) \xrightarrow{(\mathbb{T}p)^\vee_*} \pi_0
    \mathrm{Map}(\cO_{\widetilde{X}}, \LL_{\widetilde{X}}[n])
    \xrightarrow{Y^\vee_*} \pi_0 \mathrm{Map}(\cO_{\widetilde{X}},
    \cO_{\widetilde{X}}[n]) \to \cdots
  \end{equation*}
  The 1-form $\alpha_{\widetilde{X}}$ defines a class in $\pi_0
  \mathrm{Map}(\cO_{\widetilde{X}}, \LL_{\widetilde{X}}[n])$. Its
  image under $Y^\vee_*$ computes the self-pairing
  $\omega_{\widetilde{X}}(Y, Y)$. The null-homotopy
  $\omega_{\widetilde{X}}(Y, Y) \simeq 0$ given by the antisymmetry
  of $\omega_{\widetilde{X}}$ forces
  $Y^\vee_*(\alpha_{\widetilde{X}}) = 0$. By exactness of the
  sequence, the kernel of $Y^\vee_*$ equals the image of
  $(\mathbb{T}p)^\vee_*$. This guarantees the existence of a lift of
  $\alpha_{\widetilde{X}}$ to a section in $\pi_0
  \mathrm{Map}(\cO_{\widetilde{X}}, \mathbb{L}p^* \LL_X[n])$.

  \medskip

  \bfem{Step 3: Equivariant Descent.}
  To establish that the 1-form $\alpha_{\widetilde{X}}$ descends to
  the quotient stack $X$, we must determine its equivariant
  structure. The stable $\infty$-category of perfect complexes on the
  geometric realization $X \simeq \varinjlim_{\Delta^{\mathrm{op}}}
  B_\bullet(\widetilde{X}, \Gm)$ is equivalent to the stable
  $\infty$-category of $\Gm$-equivariant perfect complexes on
  $\widetilde{X}$. The $\Gm$-action lifts to the cotangent complex
  $\LL_{\widetilde{X}}$.
  By the premise of the theorem, the symplectic form
  $\omega_{\widetilde{X}}$ is an eigenform of weight 1 with respect
  to this action. Since $\Gm$ is abelian, the fundamental morphism
  $Y$ inducing the action is invariant under the adjoint action, as
  it originates from the Lie algebra of $\Gm$, giving it a weight 0
  equivariant structure. Since the equivalence $\Theta_\omega$
  intertwines the equivariant structures up to a weight 1 twist,
  post-composing the weight 0 morphism $Y$ with $\Theta_\omega$ gives
  a 1-form $\alpha_{\widetilde{X}}$ with a weight 1 equivariant structure.

  Because $\alpha_{\widetilde{X}}$ is both horizontal, as established
  in Step 2, and invariant up to a weight 1 twist, it descends to the
  quotient stack. Here the lift of Step 2 can be taken in weight one,
  since the dual triangle of Step 2 is $\Gm$-equivariant, taking
  weight-one parts is exact, and the null-homotopy of
  $\omega_{\widetilde{X}}(Y,Y)$ is canonical. Let $\cL$ be the line
  bundle on $X$ associated to the weight 1 representation of $\Gm$.
  Since $p$ is affine and $p_*\cO_{\widetilde{X}} \simeq \bigoplus_{w
  \in \mathbb{Z}} \cL^{\otimes w}$ with $\cL^{\otimes w}$ of weight
  $w$, the projection formula identifies $\Gm$-equivariant sections
  of weight $1$ of $\LL p^*E$ with sections of $E
  \otimes^{\mathbb{L}}_{\cO_X} \cL$, for $E \in L_{qcoh}(X)$. Applied
  to $E = \LL_X[n]$, this $\Gm$-equivariant horizontal global section
  descends to a global section on $X$ taking values in this twisted
  representation. This yields the descended 1-form
  \begin{equation*}
    \alpha \in \pi_0 \mathbb{R}\Gamma(X, \LL_X
    \otimes^{\mathbb{L}}_{\cO_X} \cL[n]).
  \end{equation*}

  \medskip
  \bfem{Step 4: Constructing the Contact Complex and Proving Non-Degeneracy.}
  On the quotient stack $X$, the equivalence
  $\LL_X \otimes^{\mathbb{L}}_{\cO_X} \cL[n] \simeq
  \mathbb{R}\underline{\mathcal{H}om}(\TT_X, \cL[n])$ identifies
  the descended 1-form $\alpha$ with a contraction morphism
  $\alpha^\vee \colon \TT_X \to \cL[n]$
  within $L_{qcoh}(X)$. The derived contact distribution
  $\cK \in L_{qcoh}(X)$ is defined as the homotopy fiber of this map, yielding
  the stable fiber sequence
  \begin{equation} \label{eq:contact_dist}
    \cK \to \TT_X \xrightarrow{\alpha^\vee} \cL[n]
  \end{equation}
  in $L_{qcoh}(X)$, as in Definition~\ref{berktav:contact}.

  By pulling the fiber sequence \eqref{eq:contact_dist} back to
  $\widetilde{X}$ via $\mathbb{L}p^*$,
  the twisted line bundle $\mathbb{L}p^*\cL[n]$ trivializes to
  $\cO_{\widetilde{X}}[n]$.
  The derived pullback of $\cK$, denoted $\mathbb{L}p^*\cK$, fits
  into the fiber sequence
  \begin{equation*}
    \mathbb{L}p^*\cK \to \mathbb{L}p^*\TT_X
    \xrightarrow{\mathbb{L}p^*\alpha^\vee} \cO_{\widetilde{X}}[n].
  \end{equation*}

  The underlying symplectic form provides the equivalence
  $\Theta_\omega \colon \TT_{\widetilde{X}} \xrightarrow{\sim}
  \LL_{\widetilde{X}}[n]$. The orthogonal complement $Y^\perp$ is the
  homotopy fiber of the contraction map
  $\alpha_{\widetilde{X}}^\vee \colon \TT_{\widetilde{X}} \to
  \cO_{\widetilde{X}}[n]$. As established
  in Step 2, the pairing evaluated on the fundamental morphism yields
  the path $\omega_{\widetilde{X}}(Y, Y) \simeq 0$. This path defines the
  null-homotopy $\alpha_{\widetilde{X}}^\vee \circ Y \simeq 0$.

  This datum lifts the morphism
  $Y$ to the homotopy fiber $Y^\perp$, showing $\cO_{\widetilde{X}}$
  as an isotropic object within the symplectic $\infty$-category and
  yielding the lift $i \colon \cO_{\widetilde{X}} \to Y^\perp$.

  The morphism $\alpha_{\widetilde{X}}^\vee$ factors as the
  composition $\mathbb{L}p^*\alpha^\vee \circ \mathbb{T}p$.
  The homotopy fiber of $\TT p\colon \TT_{\widetilde{X}}\to \LL
  p^*\TT_{X}$ is the relative tangent complex $\TT_{\widetilde{X}/X}$
  which is by Step 1 equivalent to $\cO_{\widetilde{X}}$. The
  homotopy fiber of $\alpha_{\widetilde{X}}^\vee$ is $Y^\perp$ by
  definition and the homotopy fiber of $\LL p^*\alpha^\vee\colon \LL
  p^* \TT_X \to \cO_{\widetilde{X}}[n]$ is the pullback of the
  contact distribution $\LL p^*\cK$. Applying Lemma
  \ref{lem:pasting_law} to $\LL p^* \alpha^\vee \circ \TT p$ yields
  the fiber sequence
  \begin{equation*}
    \cO_{\widetilde{X}}\xrightarrow{i} Y^\perp \to \LL p^*\cK
  \end{equation*}
  making $\LL p^*\cK$ the homotopy cofiber of the isotropic lift $i
  \colon \cO_{\widetilde{X}}\to Y^\perp$.

  To establish the non-degeneracy of $\LL p^*\cK$, write $V :=
  \TT_{\widetilde{X}}$ and $V/Y := \mathrm{cofib}(Y) \simeq \LL
  p^*\TT_X$ (Corollary \ref{cor:fundamental_morphism}). By the
  antisymmetry of $\omega_{\widetilde{X}}$, the composite $Y^\vee[n]
  \circ \Theta_\omega \colon V \to \cO_{\widetilde{X}}[n]$ is
  identified with $\pm\alpha_{\widetilde{X}}^\vee$, so
  $\Theta_\omega$ carries $Y^\perp$ into $\mathrm{fib}(Y^\vee[n]) =
  (V/Y)^\vee[n]$. This gives a map of fiber sequences
  \[
    \begin{tikzcd}
      Y^\perp \arrow[r] \arrow[d] & V \arrow[r,
      "\alpha_{\widetilde{X}}^\vee"] \arrow[d, "\Theta_\omega"] &
      \cO_{\widetilde{X}}[n] \arrow[d, "\pm 1"] \\
      (V/Y)^\vee[n] \arrow[r] & V^\vee[n] \arrow[r, "{Y^\vee[n]}"] &
      \cO_{\widetilde{X}}[n],
    \end{tikzcd}
  \]
  so $Y^\perp \to (V/Y)^\vee[n]$ is an equivalence if and only if
  $\Theta_\omega$ is. By construction in Step 2, the factorization
  $\alpha_{\widetilde{X}}^\vee \simeq \LL p^*\alpha^\vee \circ \TT p$
  and the lift $i$ come from the same null-homotopy of
  $\omega_{\widetilde{X}}(Y,Y)$. Dualizing the fiber sequence $\LL
  p^*\cK \to V/Y \xrightarrow{\LL p^*\alpha^\vee}
  \cO_{\widetilde{X}}[n]$ and shifting by $n$ gives a cofiber
  sequence $\cO_{\widetilde{X}} \to (V/Y)^\vee[n] \to (\LL
  p^*\cK)^\vee[n]$, and we obtain a map of cofiber sequences
  \[
    \begin{tikzcd}
      \cO_{\widetilde{X}} \arrow[r, "i"] \arrow[d, equal] & Y^\perp
      \arrow[r] \arrow[d] & \LL p^*\cK \arrow[d, "\bar\omega"] \\
      \cO_{\widetilde{X}} \arrow[r] & (V/Y)^\vee[n] \arrow[r] & (\LL
      p^*\cK)^\vee[n].
    \end{tikzcd}
  \]
  Here $\bar\omega$ is induced by the restriction of
  $\omega_{\widetilde{X}}$ to $Y^\perp$, which descends to
  $\mathrm{cofib}(i) \simeq \LL p^*\cK$ since $Y$ pairs trivially
  with $Y^\perp$. Since $\Theta_\omega$ is an equivalence, so is the
  middle vertical map, and hence so is $\bar\omega \colon \LL p^*\cK
  \xrightarrow{\sim} (\LL p^*\cK)^\vee[n]$. The same two diagrams
  show, for any $2$-form on $V$ and any null-homotopy of its value on
  $(Y,Y)$ used both for $i$ and for the factorization, that
  $\Theta_\omega$ is an equivalence if and only if $\bar\omega$ is.

  Within the mixed graded derived de Rham complex \cite[Section
  1.4.1]{CPTVV}, the derived de Rham differential is encoded by the
  mixed structure $\epsilon$. By Cartan's magic formula
  ($\mathcal{L}_Y = d_{dR} \iota_Y + \iota_Y d_{dR}$), since
  $\omega_{\widetilde{X}}$ is closed ($d_{dR}\omega_{\widetilde{X}}
  \simeq 0$) and an eigenform of weight 1 ($\mathcal{L}_Y
  \omega_{\widetilde{X}} \simeq \omega_{\widetilde{X}}$), we have
  $d_{dR}\alpha_{\widetilde{X}} = d_{dR}(\iota_Y
  \omega_{\widetilde{X}}) \simeq \omega_{\widetilde{X}}$. The
  $\Gm$-equivariant descent then ensures that the mixed differential
  of the descended 1-form $\alpha$ recovers the symplectic form upon
  pullback. Therefore, the underlying 2-form of the derived de Rham
  differential $p^*(d_{dR}\alpha)$ is precisely identified with
  $\omega_{\widetilde{X}}$.

  This equivariant non-degeneracy guarantees that the derived de Rham
  differential $d_{dR}\alpha$ induces the equivalence $\cK
  \xrightarrow{\sim} \cK^\vee \otimes^{\mathbb{L}}_{\cO_X} \cL[n]$ on
  the quotient $X$ after applying descent.

  We then complete the proof of Theorem \ref{thm:derived
  transversality}, and hence that of Theorem \ref{thm:A}.
\end{proof}

\section{Derived Legendrian Intersection Theorem}\label{sec:legendrian}
Before establishing the intersection theorem, the categorical
relationship between Legendrians and the derived symplectification
functor must be formalized.

Recall that the \bfem{${n}$-shifted contact structure} on a
derived Artin stack $X$ is defined by a contact line bundle $\cL$, an
$n$-shifted $1$-form $\alpha$, and a non-degenerate contact
distribution \[\cK \simeq \mathrm{fib}(\TT_X \xrightarrow{\alpha^\vee}
\cL[n]) \ \text{ in } L_{qcoh}(X).\]

Following \cite[Definition 4.3 and Proposition 4.4]{kib1}, the
\bfem{derived symplectification} of $X$, denoted by $\widetilde{X}$,
is defined as the total space of the $\Gm$-bundle associated with the
contact line bundle $\cL$.
The projection $p \colon \widetilde{X} \to X$ classifies
trivializations of $\cL$, forming a principal $\Gm$-bundle in
$\dArt_\K$ \cite[Section 1.3.5]{ToenHAG}.
By \cite[Theorem 4.7]{kib1} for derived schemes and \cite[Theorem
3.9]{kib2} for derived Artin stacks when $n < 0$, and by Step 1 of
the proof of Theorem~\ref{thm:Leg intersection} for all $n$,
$\widetilde{X}$ is equipped with an $n$-shifted symplectic structure
$\omega_{\widetilde{X}}$.
This symplectic form is an eigenform of weight 1 with respect to the
$\Gm$-action.

A \bfem{Legendrian morphism} $f \colon L \to X$ provides a specified
null-homotopy $f^*\alpha \simeq 0$ and induces a stable fiber sequence
\begin{equation*}
  \cO_L \to \LL_{L/X} \otimes^{\mathbb{L}}_{\cO_L} f^*\cL[n-1] \to \TT_L
\end{equation*}
of perfect complexes in $L_{qcoh}(L)$ (Definition
\ref{definition:legendrian structures}).

Derived contact geometry links these
Legendrian structures to Lagrangian structures via $\Gm$-equivariant
lifts along the symplectification projection $p$.
This lifting property translates intersection problems in contact
moduli to equivariant intersection problems in symplectic moduli, and
hence establishes the shifted contact structure on the derived
intersection of two Legendrian stacks. More specifically, we have:

\begin{theorem}  \label{thm:Leg intersection}
  Let $X$ be an $n$-shifted contact derived Artin stack. Let $L_1$ and
  $L_2$ be two derived Artin stacks equipped with Legendrian morphisms
  \[
    \begin{tikzcd}
      & L_2 \arrow[d, "f_2"] \\
      L_1 \arrow[r, "f_1"] & X.
  \end{tikzcd}\]
  Then the derived
  fiber product $Z := L_1 \times_X^h L_2$ carries an $(n-1)$-shifted
  contact structure.
\end{theorem}

\begin{proof}
  The proof proceeds in four steps utilizing the derived
  symplectification functor.

  \bfem{Step 1: Equivariant lifting of the intersecting stacks.}
  Let $\widetilde{X}$ be the derived symplectification of $X$. Let
  $\lambda_{\widetilde{X}} := p^*\alpha$, regarded as a $1$-form on
  $\widetilde{X}$ through the tautological trivialization of
  $p^*\cL$, and let $\omega_{\widetilde{X}} :=
  d_{dR}\lambda_{\widetilde{X}}$. The form $\lambda_{\widetilde{X}}$
  has weight $1$ and is horizontal, so by Cartan's formula
  $\iota_Y\omega_{\widetilde{X}} \simeq \lambda_{\widetilde{X}}$, and
  $\omega_{\widetilde{X}}$ has weight $1$. It is symplectic for every
  $n$, by the last statement of Step 4 of the proof of Theorem
  \ref{thm:derived transversality}, applied with the null-homotopy of
  $\omega_{\widetilde{X}}(Y,Y)$ given by the horizontality of
  $\lambda_{\widetilde{X}}$. Indeed, the induced form on $\LL p^*\cK$
  is the pullback of $d_{dR}\alpha|_{\cK}$, as may be checked on an
  \'etale cover $U \to X$ trivializing $\cL$, where
  $\lambda_{\widetilde{X}} = t\alpha$ and
  \[
    d_{dR}(t\alpha)|_{\cK} = \big(d_{dR}t \wedge \alpha + t\,
    d_{dR}\alpha\big)|_{\cK} = t\, d_{dR}\alpha|_{\cK},
  \]
  since $\alpha$ vanishes on $\cK$. This form is non-degenerate by
  Definition \ref{berktav:contact}. The domain stacks
  $\widetilde{L}_i := L_i \times_X^h \widetilde{X}$ are principal
  $\Gm$-bundles over $L_i$ via the projections $q_i \colon
  \widetilde{L}_i \to L_i$.

  The derived fiber product induces $\Gm$-equivariant morphisms
  $\widetilde{f}_i \colon \widetilde{L}_i \to \widetilde{X}$ in the
  homotopy pullback square
  \[
    \begin{tikzcd}
      \widetilde{L}_i \arrow[r, "\widetilde{f}_i"] \arrow[d, "q_i"'] &
      \widetilde{X} \arrow[d, "p"] \\
      L_i \arrow[r, "f_i"'] & X
    \end{tikzcd}
  \]

  The commutativity of the square induces $q_i^* f_i^* \alpha \simeq
  \widetilde{f}_i^* \lambda_{\widetilde{X}}$. Base change of the
  isotropic structure $f_i^*\alpha \simeq 0$ along $q_i$ defines a
  null-homotopy $\widetilde{f}_i^*\lambda_{\widetilde{X}} \simeq 0$
  via the diagram
  \[
    \begin{tikzcd}
      0 \arrow[r, "\sim"] \arrow[d, equal] & q_i^* f_i^* \alpha
      \arrow[d, "\sim"] \\
      0 \arrow[r, "\sim"'] & \widetilde{f}_i^* \lambda_{\widetilde{X}}
    \end{tikzcd}
  \]
  Applying the derived de Rham differential gives a $\Gm$-equivariant
  symplectic null-homotopy
  \begin{equation*}
    h_i \colon 0 \sim \widetilde{f}_i^*\omega_{\widetilde{X}} \quad
    \text{in} \quad \Acl(\widetilde{L}_i, n)^{\Gm}.
  \end{equation*}

  By Definition \ref{definition:legendrian structures} and Remark
  \ref{rmk:legendrian_sequence}, the Legendrian structure on $f_i$
  says that the sequence
  \begin{equation*}
    \TT_{L_i} \xrightarrow{\rho_i} f_i^*\cK \to \LL_{L_i}
    \otimes^{\mathbb{L}}_{\cO_{L_i}} f_i^*\cL[n],
  \end{equation*}
  whose first map is the lift of $df_i$ given by the isotropic
  structure, is a fiber sequence, that is, the induced map
  $\bar\Theta_i \colon \mathrm{cofib}(\rho_i) \to \LL_{L_i}
  \otimes^{\mathbb{L}}_{\cO_{L_i}} f_i^*\cL[n]$ is an equivalence.
  Let $W := \TT_{\widetilde{L}_i}$, $V :=
  \widetilde{f}_i^*\TT_{\widetilde{X}}$ and $g := d\widetilde{f}_i
  \colon W \to V$, and let $\eta$ be the null-homotopy of
  $\widetilde{f}_i^*\omega_{\widetilde{X}}$ underlying $h_i$. By
  \cite[Definition 2.8]{PTVV}, $h_i$ is Lagrangian if and only if the
  map $\Theta_i \colon \mathrm{cofib}(g) \to W^\vee[n]$ induced by
  $g^\vee[n] \circ \Theta_\omega$ and $\eta$ is an equivalence. Since
  $\widetilde{f}_i$ is $\Gm$-equivariant, $g$ carries the fundamental
  morphism of $\widetilde{L}_i$ to the pullback of $Y$. Contracting
  $\eta$ with it gives, by Cartan's formula, the pullback along $q_i$
  of the null-homotopy $f_i^*\alpha \simeq 0$. So $g$ lifts to
  $\widetilde{f}_i^*Y^\perp$, and on the quotients by the fundamental
  morphisms, which are $q_i^*\TT_{L_i}$ and $q_i^*f_i^*\cK$ by
  Corollary \ref{cor:fundamental_morphism} and Step 4 of the proof of
  Theorem \ref{thm:derived transversality}, the lift induces
  $q_i^*\rho_i$. On an \'etale cover trivializing $\cL$, the lifted
  path is $t h$, where $h$ is the path $f_i^*\alpha \simeq 0$, and
  $\eta = d_{dR}(th) = d_{dR}t \wedge h + t\, d_{dR}h$, where
  $d_{dR}t$ vanishes on the horizontal directions. So $\eta$ induces
  on $q_i^*\TT_{L_i}$ the pullback of the null-homotopy that defines
  $\bar\Theta_i$. The octahedral axiom for $W \to
  \widetilde{f}_i^*Y^\perp \to V$ and the dual of the Atiyah sequence
  of $q_i$ then give a map of cofiber sequences
  \[
    \begin{tikzcd}
      q_i^*\mathrm{cofib}(\rho_i) \arrow[r] \arrow[d,
      "q_i^*\bar\Theta_i"'] & \mathrm{cofib}(g) \arrow[r] \arrow[d,
      "\Theta_i"] & \cO_{\widetilde{L}_i}[n] \arrow[d, "\pm 1"] \\
      q_i^*\LL_{L_i}[n] \arrow[r] & W^\vee[n] \arrow[r] &
      \cO_{\widetilde{L}_i}[n],
    \end{tikzcd}
  \]
  under the tautological trivialization of $q_i^*f_i^*\cL$, in which
  both squares commute because all the null-homotopies involved come
  from $\eta$. By the five lemma, $\Theta_i$ is an equivalence if and
  only if $q_i^*\bar\Theta_i$ is, and the latter holds because $f_i$
  is Legendrian. Thus $h_i$ is Lagrangian, and the lifts
  $\widetilde{f}_i$ are $\Gm$-equivariant Lagrangian morphisms.

  \medskip
  \bfem{Step 2: The Transgression Loop in Equivariant Mapping Spaces.}
  The derived fiber product of the lifts is given by the homotopy
  pullback $\widetilde{Z} := \widetilde{L}_1 \times_{\widetilde{X}}^h
  \widetilde{L}_2$ \cite[Section 2.2.2]{ToenHAG}. Let $p_1 \colon
  \widetilde{Z} \to \widetilde{L}_1$ and $p_2 \colon \widetilde{Z}
  \to \widetilde{L}_2$ be the projections, and let $q \colon
  \widetilde{Z} \to \widetilde{X}$ denote the composition. Since
  homotopy pullbacks are computed within the $\infty$-category of
  $\Gm$-equivariant derived stacks, $\widetilde{Z}$ inherits a $\Gm$-action.

  Recall that the $\Gm$-action on a derived stack induces an
  auxiliary fiber grading on the derived de Rham complex relative to
  $B\Gm$ \cite[Section 2.1.1]{Calaque2019}. A form homogeneous of
  weight 1 under the $\Gm$-action carries the grading $\{1\}$.

  In our case, the symplectic form $\omega_{\widetilde{X}}$ has
  weight 1 and hence resides in the equivariant mapping space
  ${\Acl}^{\{1\}}(\widetilde{X}, n)^{\Gm}$. The $\Gm$-equivariant
  lifts $\widetilde{f}_1$ and $\widetilde{f}_2$ pull back the
  isotropic structures into the equivariant category. As
  null-homotopies, $h_1$ and $h_2$ are paths $0 \sim
  \widetilde{f}_i^*\omega_{\widetilde{X}}$ within
  ${\Acl}^{\{1\}}(\widetilde{L}_i, n)^{\Gm}$. Consequently, they
  reside in the path space $$Path_{0,
  \widetilde{f}_i^*\omega_{\widetilde{X}}}({\Acl}^{\{1\}}(\widetilde{L}_i,
  n)^{\Gm}),$$ which forms a torsor over the based loop space
  $$\Omega_0 {\Acl}^{\{1\}}(\widetilde{L}_i, n)^{\Gm} \simeq
  {\Acl}^{\{1\}}(\widetilde{L}_i, n-1)^{\Gm}.$$

  The homotopy pullback square in the $\infty$-category of
  $\Gm$-equivariant stacks provides a $\Gm$-equivariant 2-morphism
  $\gamma \colon \widetilde{f}_1 \circ p_1 \simeq \widetilde{f}_2
  \circ p_2$. Composing the null-homotopies $p_1^*h_1$ and $p_2^*h_2$
  via $\gamma$ defines a based loop $\Lambda_{\widetilde{Z}}$ at the
  zero vertex within the space of weight 1 equivariant forms
  \begin{equation*}
    \Lambda_{\widetilde{Z}} := p_1^*h_1 \cdot \gamma \cdot
    (p_2^*h_2)^{-1} \in \Omega_0 {\Acl}^{\{1\}}(\widetilde{Z}, n)^{\Gm}.
  \end{equation*}

  By \cite[Definitions 1.8 and 1.12]{PTVV}, the space of closed
  2-forms is defined, on affine pieces and then by limits, as the
  mapping space $\mathrm{Map}_{\epsilon-\mathrm{dg}_k}(\K(2)[-n-2],
  \DR(\widetilde{Z}/\K))$. The based loop space functor $\Omega_0$
  operates by pre-composing with the suspension functor $\Sigma$ on
  the domain $\K(2)[-n-2]$ \cite[Section 1.1.2]{LurieHA}. The
  suspension functor shifts the cohomological degree to
  $\K(2)[-n-1]$. It operates only on the domain complex and does not
  interact with the target complex $\DR(\widetilde{Z}/\K)$ or the
  $\Gm$-equivariant fiber grading $\{1\}$. This establishes the
  equivalence of mapping spaces
  \begin{equation*}
    \Omega_0 {\Acl}^{\{1\}}(\widetilde{Z}, n)^{\Gm} \simeq
    {\Acl}^{\{1\}}(\widetilde{Z}, n-1)^{\Gm}.
  \end{equation*}
  Let $\omega_{\widetilde{Z}}$ be the image of
  $\Lambda_{\widetilde{Z}}$ under this equivalence, so that it
  resides in this target space. It hence inherently carries the
  $\{1\}$ grading and remains a weight 1 eigenform.

  \medskip
  \bfem{Step 3: Verification of Non-Degeneracy.}
  By the main theorem of shifted symplectic intersections
  \cite[Theorem 2.9]{PTVV}, the derived fiber product of Lagrangian
  morphisms carries an $(n-1)$-shifted symplectic structure.
  Establishing this non-degeneracy equivariantly requires verifying
  that the induced morphism $\Theta_{\omega_{\widetilde{Z}}} \colon
  \TT_{\widetilde{Z}} \xrightarrow{\sim} \LL_{\widetilde{Z}}[n-1]$ is
  an equivalence.

  The homotopy pullback square gives a stable fiber sequence of
  relative tangent complexes
  \begin{equation*}
    \TT_{\widetilde{Z}/\widetilde{L}_2} \to \TT_{\widetilde{Z}} \to
    p_2^*\TT_{\widetilde{L}_2}
  \end{equation*}
  in $L_{qcoh}(\widetilde{Z})$ \cite[Lemma 1.4.1.16]{ToenHAG}.

  Derived base change along the projection $p_2$ gives the
  equivalence $\TT_{\widetilde{Z}/\widetilde{L}_2} \simeq
  p_1^*\TT_{\widetilde{L}_1/\widetilde{X}}$. Substituting this into
  the sequence we obtain
  \begin{equation*}
    p_1^*\TT_{\widetilde{L}_1/\widetilde{X}} \to \TT_{\widetilde{Z}}
    \to p_2^*\TT_{\widetilde{L}_2}.
  \end{equation*}
  Conversely, we consider the relative cotangent fiber sequence for
  the projection $p_1$
  \begin{equation*}
    p_1^*\LL_{\widetilde{L}_1}[n-1] \to \LL_{\widetilde{Z}}[n-1] \to
    \LL_{\widetilde{Z}/\widetilde{L}_1}[n-1].
  \end{equation*}
  Again, derived base change yields
  $\LL_{\widetilde{Z}/\widetilde{L}_1} \simeq
  p_2^*\LL_{\widetilde{L}_2/\widetilde{X}}$. The symplectic form
  $\omega_{\widetilde{X}}$ and the Lagrangian isotropic structures
  $h_1, h_2$ induce a morphism between these two fiber sequences. The
  Lagrangian condition on $\widetilde{L}_2$ requires that the induced
  map $\Theta_2 \colon \TT_{\widetilde{L}_2} \xrightarrow{\sim}
  \LL_{\widetilde{L}_2/\widetilde{X}}[n-1]$ is an equivalence in
  $L_{qcoh}(\widetilde{L}_2)$ \cite[Definition 2.8]{PTVV}.

  Rotating the exact triangle for the Lagrangian immersion
  $\widetilde{L}_1 \to \widetilde{X}$ and applying the $\cO$-linear
  duality functor along with a cohomological shift shows the
  isotropic structure $h_1$ equivalently induces an isomorphism
  $\Theta_1^\vee[n-1] \colon \TT_{\widetilde{L}_1/\widetilde{X}}
  \xrightarrow{\sim} \LL_{\widetilde{L}_1}[n-1]$.

  Pulling these back to $\widetilde{Z}$ constructs the commutative
  diagram of exact triangles
  \[
    \begin{tikzcd}
      p_1^*\TT_{\widetilde{L}_1/\widetilde{X}} \arrow[r] \arrow[d,
      "p_1^*\Theta_1^\vee{[n-1]} \simeq"'] &
      \TT_{\widetilde{Z}} \arrow[r] \arrow[d,
      "\Theta_{\omega_{\widetilde{Z}}}"] &
      p_2^*\TT_{\widetilde{L}_2} \arrow[d, "p_2^*\Theta_2 \simeq"'] \\
      p_1^*\LL_{\widetilde{L}_1}[n-1] \arrow[r] &
      \LL_{\widetilde{Z}}[n-1] \arrow[r] &
      p_2^*\LL_{\widetilde{L}_2/\widetilde{X}}[n-1]
    \end{tikzcd}
  \]
  in $L_{qcoh}(\widetilde{Z})$. This diagram commutes up to the
  homotopies provided by the paths $h_1, h_2$, and the 2-cell
  $\gamma$. By the pasting law for stable fiber sequences (Lemma
  \ref{lem:pasting_law}), since the outer vertical morphisms are
  equivalences of perfect complexes, the induced central morphism
  $\Theta_{\omega_{\widetilde{Z}}}$ is an equivalence. Therefore,
  $\omega_{\widetilde{Z}}$ is a {\em $\Gm$-equivariant
  $(n-1)$-shifted symplectic form.}

  \medskip
  \bfem{Step 4: Descent via Base Change.}
  Let $\pi \colon \widetilde{Z} \to Z$ be the projection to the
  contact stack quotient $Z := L_1 \times_X^h L_2$. To apply Theorem
  \ref{thm:A}, the projection $\pi$ must be a principal $\Gm$-bundle
  over $Z$. The homotopy pullback defining $\widetilde{Z}$ evaluates as
  \begin{equation*}
    \widetilde{Z} := \widetilde{L}_1 \times_{\widetilde{X}}^h
    \widetilde{L}_2 \simeq (L_1 \times_X^h \widetilde{X})
    \times_{\widetilde{X}}^h (L_2 \times_X^h \widetilde{X}).
  \end{equation*}
  By the associativity and commutativity of limits in $\dArt_\K$,
  this simplifies to
  \begin{equation*}
    \widetilde{Z} \simeq (L_1 \times_X^h L_2) \times_X^h
    \widetilde{X} \simeq Z \times_X^h \widetilde{X}.
  \end{equation*}

  The descent structure is given by the Cartesian cube. Denote by
  $q_1 \colon \widetilde{L}_1 \to L_1$ and $q_2 \colon
  \widetilde{L}_2 \to L_2$ the principal $\Gm$-bundles associated
  with the lifts. The commutativity of the Lagrangian morphisms
  $\widetilde{f}_i$ with the torsor projections is given by the
  base-change 2-cell equivalences $\beta_1$ and $\beta_2$.

  Here, the top face encodes the Lagrangian intersection equipped
  with the 2-morphism $\gamma \colon \widetilde{f}_1 \circ p_1 \simeq
  \widetilde{f}_2 \circ p_2$.
  \[
    \begin{tikzcd}[column sep={1.3cm,between origins}, row
      sep={1.3cm,between origins}]
      & \widetilde{Z} \arrow[rrr, "{p_2}"] \arrow[ld, "{p_1}"']
      \arrow[dd, "{\pi}"' near start]
      & & & \widetilde{L}_2 \arrow[ld, "{\widetilde{f}_2}"]
      \arrow[dd, "{q_2}"] \\
      \widetilde{L}_1 \arrow[rrr, "{\widetilde{f}_1}"' pos=0.75,
      crossing over] \arrow[dd, "{q_1}"']
      & & & \widetilde{X}
      & \\
      & Z \arrow[rrr, "{\pi_2}" pos=0.3] \arrow[ld, "{\pi_1}"']
      & & & L_2 \arrow[ld, "{f_2}"] \\
      L_1 \arrow[rrr, "{f_1}"']
      & & & X
      \arrow[from=2-4, crossing over, "{p}" near start]
      &
      % \gamma (Top face)
      \arrow[from=1-2, to=2-4, phantom, "{\simeq\gamma}" description]
      % \bar{\gamma} (Bottom face)
      \arrow[from=3-2, to=4-4, phantom, "{\simeq\bar\gamma}"
      description, pos=0.22]
      % \beta_1 (Front face) - Restored to 0.86 to sit clearly in the
      % front panel!
      \arrow[from=2-1, to=4-4, phantom, "{\simeq\beta_1}" description, pos=0.86]
      % \beta_2 (Right face)
      \arrow[from=1-5, to=4-4, phantom, "{\simeq\beta_2}" description, pos=0.4]
    \end{tikzcd}
  \]

  Since $\widetilde{X} \to X$ is a $\Gm$-torsor, and torsors are
  stable under derived base change \cite[Section 1.3.5]{ToenHAG}, the
  projection $\pi \colon \widetilde{Z} \to Z$ is a principal
  $\Gm$-bundle over $Z$. The contact line bundle on the intersection
  $Z$ is defined as $\cL_Z := \pi_1^* f_1^* \cL \simeq \pi_2^* f_2^*
  \cL$, where $\pi_i \colon Z \to L_i$ are the projections. Since
  $\widetilde{Z} \simeq Z \times_X^h \widetilde{X}$, we conclude that
  $\widetilde{Z}$ is the principal $\Gm$-bundle associated with the
  line bundle $\cL_Z$ on $Z$.

  By Step 2 and Step 3, $(\widetilde{Z}, \omega_{\widetilde{Z}})$ is
  an $(n-1)$-shifted symplectic stack equipped with a weight 1
  $\Gm$-action. Moreover, $\pi$ is trivial on an \'etale cover of
  $Z$, since $\widetilde{Z} \simeq Z \times_X^h \widetilde{X}$ and
  the principal $\Gm$-bundle $\widetilde{X} \to X$ is trivial on an
  \'etale cover of $X$. Applying Theorem \ref{thm:A}, the stack
  quotient $Z \simeq \varinjlim_{\Delta^{\mathrm{op}}}
  B_\bullet(\widetilde{Z}, \Gm)$ descends the data to an
  $(n-1)$-shifted contact structure on $Z$. The descent of the
  top-face coherence $\gamma$ along the torsor projections and
  lateral base-change equivalences $\beta_1, \beta_2$ uniquely
  defines the Legendrian 2-morphism $\bar{\gamma} \colon f_1 \circ
  \pi_1 \simeq f_2 \circ \pi_2$.

  The fundamental vector field of the $\Gm$-action pairs with
  $\omega_{\widetilde{Z}}$ to generate the 1-form
  $\alpha_{\widetilde{Z}}$ on the total space. Equivariant descent
  along the torsor projection $\pi$ yields the 1-form $\alpha_Z \in
  \pi_0\mathbb{R}\Gamma(Z, \LL_Z \otimes^{\mathbb{L}}_{\cO_Z} \cL_Z[n-1])$.

  The contact distribution $\cK_Z$ is then the homotopy fiber
  $\mathrm{fib}(\TT_Z \to \cL_Z[n-1])$. The differential $d_{dR}$
  commutes with the quotient map $\pi$. The underlying 2-form of
  $d_{dR}\alpha_Z$ pulls back to $\omega_{\widetilde{Z}}$. From Step
  3, $\omega_{\widetilde{Z}}$ induces an equivalence
  $\TT_{\widetilde{Z}} \xrightarrow{\sim} \LL_{\widetilde{Z}}[n-1]$.
  Descent of this equivariant non-degeneracy establishes the
  equivalence $\cK_Z \xrightarrow{\sim} \cK_Z^\vee
  \otimes^{\mathbb{L}}_{\cO_Z} \cL_Z[n-1]$ in $L_{qcoh}(Z)$.

  This concludes the proof of Theorem \ref{thm:Leg intersection}, and
  consequently, also proves Theorem \ref{thm:B}.

\end{proof}

\section{Applications} \label{sec:moduli}
\subsection{\texorpdfstring{$(-1)$}{(-1)}-Shifted Contact Structures
via Derived Legendrian Intersections}\label{sec: (-1)-shifted contact
via derived leg int}

In the case of smooth schemes, an immediate corollary of Theorem
\ref{thm:Leg intersection} arises as follows.
\begin{corollary}\label{cor:smooth_legendrians}
  Let $X$ be a smooth $\K$-scheme with a contact form $\alpha$ with
  values in a line bundle $\cL$, and let $L, L' \subset X$ be smooth
  closed subschemes of dimension $(\dim X - 1)/2$ on which $\alpha$ vanishes.
  Then the derived fiber product
  \(
    L \times_X^h L'
  \)
  carries a canonical $(-1)$-shifted contact structure.
\end{corollary}

\begin{proof}
  The vanishing of $\alpha$ on $L$ gives an isotropic structure.
  Since $L$ is smooth, $$\LL_{L/X} \simeq N^\vee_{L/X}[1],$$ where
  $N_{L/X}$ is the normal bundle. The contact distribution $\xi :=
  \ker(\alpha)$ restricts to a vector bundle on $L$ containing $T_L$,
  and $T_L$ is Lagrangian for $d_{dR}\alpha|_{\xi}$, since it is
  isotropic of half the rank. Hence $d_{dR}\alpha$ identifies
  $\xi|_L/T_L$ with $T_L^\vee \otimes \cL|_L$, and $\alpha$
  identifies $T_X|_L/\xi|_L$ with $\cL|_L$. This gives an exact
  sequence $0 \to T_L^\vee \otimes \cL|_L \to N_{L/X} \to \cL|_L \to
  0$, and the sequence of Definition \ref{definition:legendrian
  structures} is its dual tensored with $\cL|_L$, the exact sequence
  $0 \to \cO_L \to N^\vee_{L/X} \otimes \cL|_L \to T_L \to 0$. So
  $L$, and likewise $L'$, is Legendrian. Since $\cL$ is Zariski
  locally trivial on the scheme $X$, Theorem \ref{thm:Leg intersection} applies.

\end{proof}

We will now examine a special and interesting case:
\begin{example}[Derived discriminant locus]\label{example: ddislocus}
  Given any smooth scheme $L$ over $\K$, every regular function $f\in
  \cO(L)$ gives a\textit{ Legendrian embedding} of $L$ into the 1-jet
  space $(J^1(L), \xi_{jet})$ via the map \cite[Example
  2.5.11]{Geiges} $$F:L\longrightarrow \mathbb{A}^1_\K \times T^*L,
  \quad p \mapsto (f(p),d_{dR}f_p).$$ Here,  we have the
  identification $J^1(L)\simeq \mathbb{A}^1_\K \times T^*L$ such that
  the 1-jet space $J^1(L)$ carries a canonical (0-shifted) contact
  structure $\xi_{jet}$. Note in particular that the zero function
  $f\equiv0$ corresponds to the zero section $j^10 \simeq L \subset
  T^*L \subset T^*L \times \mathbb{A}^1_\K \simeq J^1(L)$ of  $J^1(L).$

  Classically, we define the (ordinary) \textit{discriminant locus}
  to be the image of the intersection of the graph $\Gamma_F$ and the
  zero-section of the 1-jet bundle $J^1(L),$ which usually fails to
  be a smooth scheme due to the expected transversality and degeneracy issues.

  We then consider the \bfem{derived discriminant locus}
  $\Delta\mathrm{loc}(f)$ of $f$ defined as the derived fiber product
  \[
    \begin{tikzcd}
      \Delta\mathrm{loc}(f):= L \times_{J^1(L)}^h L   \arrow[r]
      \arrow[d]             & L \arrow[d,"{0}"] \\
      L \arrow[r, "F"] & (J^1(L), \xi_{jet}).
  \end{tikzcd}\]

  The contact form $dz - \xi$ of $J^1(L)$ vanishes on the embeddings
  $F$ and $0$, which are smooth closed of dimension $\dim L$, so they
  are Legendrian by the proof of Corollary
  \ref{cor:smooth_legendrians}. Applying Theorem \ref{thm:Leg
  intersection} to the Legendrian morphisms $F: L\rightarrow (J^1(L),
  \xi_{jet}) \leftarrow L :{0}$, we conclude:
  \begin{corollary}
    The derived discriminant locus
    $\Delta\mathrm{loc}(f)$ of a regular function $f$ on a smooth
    $\K$-scheme $L$ admits a $(-1)$-shifted contact structure.

  \end{corollary}
\end{example}

\subsection{Projective Higgs Bundles and Legendrian Intersections}
Let $X$ be a smooth proper variety of dimension $d$ over $\K$ and let
$G$ be a reductive algebraic group. The derived moduli stack of
principal $G$-bundles $\mathrm{Bun}_G(X)$ is a derived Artin stack.
By Serre duality, the cotangent complex $\LL_{\mathrm{Bun}_G(X)}$ at
a point $P$ is equivalent to $R\Gamma(X, \mathrm{ad}(P)^* \otimes
\omega_X)[d-1]$. To corepresent classical $\omega_X$-valued Higgs
fields $\theta \in H^0(X, \mathrm{ad}(P)^* \otimes \omega_X)$ in
cohomological degree 0, the cotangent stack must be shifted by $1-d$.
Thus, the derived moduli space of $G$-Higgs bundles on $X$ is the
$(1-d)$-shifted cotangent stack $T^*[1-d]\mathrm{Bun}_G(X)$. Hence,
it carries a canonical $(1-d)$-shifted symplectic structure.

The multiplicative group $\Gm$ acts on $T^*[1-d]\mathrm{Bun}_G(X)$ by
scaling the cotangent fibers, and the weight 1 condition holds. Let
$T^*[1-d]\mathrm{Bun}_G(X)^\circ$ denote the open substack where the
Higgs field vanishes nowhere. The principal $\Gm$-bundle over its
quotient is not trivial on any \'etale cover near the Higgs bundles
that are fixed by scaling up to automorphisms.

For instance, for $G = \mathrm{SL}_2$ and $d = 1$, the uniformizing
Higgs bundle $K^{1/2} \oplus K^{-1/2}$, whose Higgs field is the
identity $K^{1/2} \to K^{-1/2} \otimes K$, is such a point, since the
automorphisms $\mathrm{diag}(\mu,\mu^{-1})$ absorb the scaling by
$\mu^{-2}$. So Theorem \ref{thm:derived transversality} does not
apply, and we use the symplectification instead. The Liouville form
of $T^*[1-d]\mathrm{Bun}_G(X)$ has weight $1$, and its de Rham
differential is the symplectic structure \cite[Section 2.1.1 and
Theorem 2.4]{Calaque2019}. So $[T^*[1-d]\mathrm{Bun}_G(X)/\Gm]$ is
$(1-d)$-shifted contact in the sense of the symplectification
(Remark~\ref{rmk:etale}), and so is its open substack below.

\begin{corollary} \label{cor:projective_higgs}
  The derived moduli stack of projective Higgs bundles
  \[
    P\mathrm{Higgs}_G(X) := [T^*[1-d]\mathrm{Bun}_G(X)^\circ / \Gm]
  \]
  admits a canonical $(1-d)$-shifted contact structure in the sense
  of the symplectification (Remark~\ref{rmk:etale}).
\end{corollary}

We restrict to the case $d=1$, where $C$ is a smooth proper curve.
Substituting $d=1$ into Corollary \ref{cor:projective_higgs} yields
the following specialization.

\begin{corollary} \label{cor:curve_higgs}
  The derived moduli stack of projective Higgs bundles
  $P\mathrm{Higgs}_G(C)$ admits a canonical $0$-shifted contact
  structure in the sense of the symplectification (Remark~\ref{rmk:etale}).
\end{corollary}

This ambient space serves as the geometric foundation to evaluate
negative-shifted derived intersections.
Consider the global nilpotent cone $\mathcal{N} \subset
T^*[0]\mathrm{Bun}_G(C)$, which is the central fiber of the Hitchin
fibration. The nilpotency condition is invariant under the
$\Gm$-scaling action. The projectivization of an irreducible
component of the nilpotent cone, denoted $\mathbb{P}\mathcal{N}_i$,
descends to a well-defined substack in $P\mathrm{Higgs}_G(C)$. For $G
= \mathrm{GL}_n$, the generic smooth part of each irreducible
component of the reduced nilpotent cone is an isotropic substack of
the $0$-shifted symplectic Higgs moduli space
\cite[Theorem~5.9(2)]{BenBassatDasPantev}, and since its dimension is
half the dimension of the ambient space \cite[proof of
Theorem~5.9(3)]{BenBassatDasPantev}, it is Lagrangian in the
classical sense. A Legendrian structure on $\mathbb{P}\mathcal{N}_i$
in the sense of the symplectification (Remark~\ref{rmk:etale})
requires more, namely a $\Gm$-equivariant, weight-one Lagrangian
structure on a derived enhancement of the component, which we do not
construct here.

\begin{corollary} \label{cor:nilpotent_intersection}
  Let $\mathbb{P}\mathcal{N}_1$ and $\mathbb{P}\mathcal{N}_2$ be
  Legendrian substacks in $P\mathrm{Higgs}_G(C)$, in the sense of the
  symplectification (Remark~\ref{rmk:etale}), associated with the
  projectivizations of two distinct irreducible components of the
  global nilpotent cone. Their derived fiber product
  \[
    \mathbb{P}\mathcal{N}_1 \times_{P\mathrm{Higgs}_G(C)}^h
    \mathbb{P}\mathcal{N}_2
  \]
  carries a canonical $(-1)$-shifted contact structure in the sense
  of the symplectification (Remark~\ref{rmk:etale}).
\end{corollary}

\begin{proof}
  The inclusions $f_i \colon \mathbb{P}\mathcal{N}_i \to
  P\mathrm{Higgs}_G(C)$ are Legendrian morphisms into a $0$-shifted
  contact stack. The homotopy pullback of their $\Gm$-equivariant
  Lagrangian lifts in the derived symplectification
  $T^*[0]\mathrm{Bun}_G(C)^\circ$ yields a $\Gm$-equivariant
  $(-1)$-shifted symplectic structure of weight 1. By Steps 2 and 3
  of the proof of Theorem \ref{thm:Leg intersection}, which only use
  that the lifts are Lagrangian with isotropic structures of weight
  $1$, and by the identification of $\widetilde{Z}$ with the
  principal $\Gm$-bundle over the intersection in Step 4, this is a
  $(-1)$-shifted contact structure on the intersection in the sense
  of the symplectification (Remark~\ref{rmk:etale}).
\end{proof}

\subsection{\texorpdfstring{$\ell$}{l}-adic Local Systems and
Frobenius}\label{subsec:ladic}
We thank the anonymous referee for suggesting this direction to
promote a contact moduli space from arithmetic geometry.

In many examples, a group action preserves a shifted symplectic
structure only up to a character, and the quotient is then no longer
shifted symplectic. The descent in Step 3 of the proof of Theorem
\ref{thm:derived transversality} still applies, and it gives the
quotient a symplectic structure twisted by the line bundle of the
character. For a discrete group, this twisted structure plays the
role that the contact structure plays for $\Gm$. For a discrete group
$\Gamma$ acting on a derived stack $M$ and a character $\chi \colon
\Gamma \to \K^\times$, let $\cL_\chi$ be the line bundle on
$[M/\Gamma]$ associated to $\chi$. A closed $2$-form of degree $n$ on
$[M/\Gamma]$ with values in $\cL_\chi$ is a $\Gamma$-homotopy fixed
point of $\Acl(M,n)$ for the action $\gamma \cdot \beta :=
\chi(\gamma)^{-1}\gamma^*\beta$. For $\Gamma = \mathbb{Z}$ with
generator $F$, this amounts to a single equivalence $F^*\beta \simeq
\chi(F)\,\beta$.

\begin{proposition}\label{prop:discrete_descent}
  Let $\Gamma$ be a discrete group acting on an $n$-shifted
  symplectic derived Artin stack $(M,\omega)$, and let $\chi \colon
  \Gamma \to \K^\times$ be a character. Suppose that $\omega$ is a
  $\Gamma$-homotopy fixed point for the action $\gamma \cdot \beta =
  \chi(\gamma)^{-1}\gamma^*\beta$. Then $\omega$ descends to a closed
  $2$-form $\bar\omega$ of degree $n$ on $[M/\Gamma]$ with values in
  $\cL_\chi$, and $\bar\omega$ is non-degenerate, i.e., it induces an
  equivalence $\TT_{[M/\Gamma]} \simeq \LL_{[M/\Gamma]}
  \otimes^{\mathbb{L}} \cL_\chi[n]$.
\end{proposition}

\begin{proof}
  The quotient $[M/\Gamma]$ is the colimit of the groupoid $M \times
  \Gamma^{\bullet}$, and closed forms send colimits of derived stacks
  to limits, since they are maps into the derived stack $\Acl(-,n)$
  \cite[Definition 1.12]{PTVV}. Hence $\cL_\chi$-valued closed
  $2$-forms of degree $n$ on $[M/\Gamma]$ are the $\Gamma$-homotopy
  fixed points of $\Acl(M,n)$ for the twisted action, and $\omega$
  descends to such a form $\bar\omega$. Since $\Gamma$ is discrete,
  the projection $M \to [M/\Gamma]$ is \'etale. So the pullback of
  $\TT_{[M/\Gamma]}$ to $M$ is $\TT_M$, and the pullback of
  $\cL_\chi$ to $M$ is trivial. Under these identifications, the
  pullback of the map $\TT_{[M/\Gamma]} \to \LL_{[M/\Gamma]}
  \otimes^{\mathbb{L}} \cL_\chi[n]$ induced by $\bar\omega$ is the
  equivalence $\TT_M \simeq \LL_M[n]$ induced by $\omega$. Since
  equivalences can be checked after pullback along the \'etale cover
  $M \to [M/\Gamma]$, the form $\bar\omega$ is non-degenerate.
\end{proof}

\begin{remark}\label{rmk:discrete_vs_Gm}
  If the action of $\Gamma$ extends to a $\Gm$-action of weight $1$
  through $\chi$, then $[M/\Gamma] \simeq [M/\Gm] \times_{B\Gm}
  B\Gamma$, and $\omega$ is a contact structure on $[M/\Gm]$ in the
  sense of the symplectification (Remark~\ref{rmk:etale}). For a
  discrete group alone, there is no fundamental vector field as in
  Step 2 of the proof of Theorem \ref{thm:derived transversality},
  and the descended structure is the twisted symplectic structure.
\end{remark}

We now apply this to Frobenius. Let $X$ be a smooth projective
geometrically connected variety of dimension $d$ over a finite field
$\mathbb{F}_q$, let $\bar X := X \times_{\mathbb{F}_q}
\bar{\mathbb{F}}_q$, and let $\ell$ be a prime not dividing $q$. The
derived moduli stack $LocSys_{\ell,r}(\bar X)$ of rank $r$
$\ell$-adic local systems on $\bar X$ is a derived
$\mathbb{Q}_\ell$-analytic stack, and \'etale Poincar\'e duality
equips it with a $(2-2d)$-shifted symplectic structure $\omega$
\cite[Theorem 5.13 and the paragraph after it]{Antonio}. At a local
system $\rho$, the underlying pairing on the tangent complex
$C^*_{et}(\bar X, \mathrm{Ad}(\rho))[1]$ is the cup product followed
by $H^{2d}_{et}(\bar X, \mathbb{Q}_\ell)[2-2d] \simeq
\mathbb{Q}_\ell[2-2d]$. The last identification is the trace map
$H^{2d}_{et}(\bar X, \mathbb{Q}_\ell(d)) \simeq \mathbb{Q}_\ell$
composed with a choice of isomorphism $\mathbb{Q}_\ell(d) \simeq
\mathbb{Q}_\ell$, which exists over $\bar{\mathbb{F}}_q$.

Let $F \colon \bar X \to \bar X$ be the base change of the $q$-power
Frobenius endomorphism of $X$ over $\mathbb{F}_q$, which is the
$\nu$-fold relative Frobenius of $X$ over $\mathbb{F}_q$ for $q =
p^\nu$ \cite[Expos\'e XV, \S 1, n\textsuperscript{o} 4]{SGA5}. It is
finite locally free of rank $q^d$, as an iterate of the relative
Frobenius of a smooth morphism of relative dimension $d$ \cite[Tags
07QB and 0FW2]{stacks-project}, and it is integral, surjective and
radicial \cite[Expos\'e XV, \S 1, n\textsuperscript{o} 2, Proposition
2 a)]{SGA5}, so pullback along $F$ is an equivalence of \'etale sites
\cite[Expos\'e VIII, Th\'eor\`eme 1.1]{SGA4-2} and hence induces an
automorphism of $LocSys_{\ell,r}(\bar X)$.

The pullback $F^*$ commutes with cup products, and on
$H^{2d}_{et}(\bar X, \mathbb{Q}_\ell) \simeq \mathbb{Q}_\ell$ it is
multiplication by the degree $q^d$, by the compatibility of the trace
map with composition and its normalization for finite locally free
morphisms \cite[Expos\'e XVIII, Th\'eor\`eme 2.9, (Var 3) and (Var
4)(I)]{SGA4-3}. Hence the symplectic structure satisfies
\begin{equation*}
F^*\omega \simeq q^d\,\omega.
\end{equation*}
So $\omega$ is preserved by $F$ only up to the character $\chi \colon
\mathbb{Z} \to \mathbb{Q}_\ell^\times$ with $\chi(F) = q^d$, and the
quotient by $F$ is not shifted symplectic. The proof of Proposition
\ref{prop:discrete_descent} only uses that closed forms on the
quotient are homotopy fixed points and that non-degeneracy can be
checked on an \'etale cover, so it applies to derived
$\mathbb{Q}_\ell$-analytic stacks as well. We obtain the following.

\begin{corollary} \label{cor: l-adic local sys}
The quotient $[LocSys_{\ell,r}(\bar X)/\mathbb{Z}]$ by the action of
Frobenius carries a canonical $(2-2d)$-shifted symplectic structure
with values in the line bundle $\cL_\chi$, where $\chi(F) = q^d$.
\end{corollary}

By Remark~\ref{rmk:discrete_vs_Gm}, this structure is the base change
of a contact structure whenever the Frobenius action extends to a
$\Gm$-action of weight $1$. For $d$ odd, however, the quotient
carries no $(2-2d)$-shifted contact structure at all. Since $2-2d$ is
then divisible by $4$, shifting the tangent complex by the even
integer $d-1$ turns a $(2-2d)$-shifted symplectic structure into a
non-degenerate antisymmetric pairing of degree $0$. On each fiber,
such a pairing matches the parts of degrees $i$ and $-i$ and is
symplectic on the part of degree $0$, so the Euler characteristic of
each fiber is even. Hence $LocSys_{\ell,r}(\bar X)$ has even virtual
dimension, and so does its quotient, since $LocSys_{\ell,r}(\bar X)
\to [LocSys_{\ell,r}(\bar X)/\mathbb{Z}]$ is \'etale. The same
argument applied to the contact distribution shows that
$(2-2d)$-shifted contact stacks have odd virtual dimension. So the
Frobenius action gives the twisted symplectic structure of
Corollary~\ref{cor: l-adic local sys} in place of a contact structure.

\subsection{Projective Coadjoint Stacks and Lie 2-Groups}

Connective structures on multiplicative gerbes define
prequantizations of $BG$, yielding 2-shifted contact structures on
the classifying stack of the associated Lie 2-group $\mathfrak{G}$
\cite[Section~3.2.4]{Tellez}. That construction relies on principal
$B^2\Gm$-fibrations. We arrange a parallel geometry within our setup
of derived symplectification via $\Gm$-torsors.

Consider the classifying stack $BG$ of a reductive algebraic group
$G$. The 2-shifted cotangent stack $T^*[2]BG \simeq
\mathfrak{g}^\vee[1]/G$ carries a canonical 2-shifted symplectic
structure  $\omega$ thanks to Theorem \ref{thm:calaque}. The
multiplicative group $\Gm$ acts on the shifted fibers
$\mathfrak{g}^\vee[1]$ by scaling. This action dilates the symplectic
form $\omega$ with weight 1.

Since $\mathfrak{g}^\vee[1]$ is concentrated in cohomological degree
$-1$, the zero section of $T^*[2]BG$ contains all of its points, so
its complement is empty. We therefore keep the zero section and use
the symplectification. The Liouville form of $T^*[2]BG$ has weight
$1$, and its de Rham differential is the symplectic structure
\cite[Section 2.1.1 and Theorem 2.4]{Calaque2019}, which gives the
following result.

\begin{corollary} \label{cor: Lie 2-group}
The derived stack quotient
\begin{equation*}
  \mathbb{P}(T^*[2]BG) := [T^*[2]BG / \Gm]
\end{equation*}
inherits a canonical 2-shifted contact structure in the sense of the
symplectification (Remark~\ref{rmk:etale}).
\end{corollary}

This projective stack provides an algebraic prequantization of $BG$.
Its contact line bundle is not trivial on any \'etale cover, since
the scaling $\Gm$ in the stabilizer of each point acts on its fibers
with weight one (Remark~\ref{rmk:etale}). Hence Theorem \ref{thm:Leg
intersection} does not apply, and we use Steps 2 and 3 of the proof
of Theorem \ref{thm:Leg intersection} instead, which only use that
the lifts are Lagrangian with isotropic structures of weight $1$,
together with the identification of the principal $\Gm$-bundle in
Step 4. Recall that for a closed subgroup $H \subset G$, the
2-shifted conormal stack $\mathcal{N}^*_{BH/BG}[2]$ is Lagrangian in
$T^*[2]BG$ \cite[Theorem 2.13]{Calaque2019}. Its isotropic structure
comes from the factorization of $\mathcal{N}^*_{BH/BG}[2] \to
T^*[2]BH$ through the zero section \cite[Section 2.3]{Calaque2019},
and this factorization is $\Gm$-equivariant for the trivial action on
$BH$. Since $\Gm$ acts trivially on $BH$, the forms of weight one on
$BH$ vanish, so this isotropic structure has weight one. Therefore
its projectivization $\mathbb{P}(\mathcal{N}^*_{BH/BG}[2]) :=
[\mathcal{N}^*_{BH/BG}[2]/\Gm] \to \mathbb{P}(T^*[2]BG)$ is a
Legendrian morphism in the sense of the symplectification
(Remark~\ref{rmk:etale}).

\par From these steps, we then conclude:
\begin{corollary}\label{cor: intersection of prequan}
The derived fiber product of two such Legendrians
$$\mathbb{P}(\mathcal{N}^*_{BH_1/BG}[2]) \to \mathbb{P}(T^*[2]BG)
\leftarrow \mathbb{P}(\mathcal{N}^*_{BH_2/BG}[2])$$ associated with
subgroups $H_1$ and $H_2$ carries a canonical 1-shifted contact
structure in the sense of the symplectification (Remark~\ref{rmk:etale}).
\end{corollary}

\section{Concluding Remarks}

In this paper, we have adapted the traditional transversality lemma
from contact geometry to the framework of derived algebraic geometry
and have shown the derived Legendrian intersection theorem, along
with numerous applications, including the derived geometry of the
discriminant loci of 1-jet bundles and certain moduli problems.

In brief, Theorem \ref{thm:derived transversality} establishes the
categorical counterpart of the previously mentioned lemma for derived
Artin stacks. In the classical setting, the transversality
requirement (transverse hypersurface-local slice) arises because the
global orbit space is often not a smooth, Hausdorff manifold. The
derived stack quotient, on the other hand, includes the orbits that
prevent the orbit space from being a smooth, Hausdorff manifold, so
no local slice is needed and the contact reduction is global.

Theorem \ref{thm:Leg intersection} proves the derived Legendrian
intersection theorem via base change and an $\infty$-categorical
descent cube, along with a certain lifting property arising from the
symplectic-contact dictionary. As shown in the proof of Theorem
\ref{thm:Leg intersection}, this lifting property translates
intersection problems in contact
moduli to equivariant intersection problems in symplectic moduli, and
hence establishes the desired result.

As a first application of the Legendrian intersection theorem,
Theorem \ref{thm:Leg intersection}, we study in Example \ref{example:
ddislocus} the derived geometry of the discriminant loci of 1-jet
bundles and show that these loci possess a $(-1)$-shifted contact structure.

In the last section, we present further applications of Theorem
\ref{thm:derived transversality} and Theorem \ref{thm:Leg
intersection}, of Proposition \ref{prop:discrete_descent}, to moduli
theory, see  Corollary \ref{cor:projective_higgs}, Corollary
\ref{cor: l-adic local sys}, Corollary \ref{cor: Lie 2-group}, and
Corollary \ref{cor: intersection of prequan}, focusing on the moduli
spaces of projective Higgs bundles, $\ell$-adic local systems, and Lie 2-groups.

\section*{Acknowledgments}
Kadri \.{I}lker Berktav posed the initial questions of formulating
the transversality lemma and Legendrian intersections in the derived
setup. He supervised the research, provided ideas for the proof of
the Legendrian intersection theorem (Theorem \ref{thm:Leg
intersection}), suggested directions for the moduli applications, and
contributed the derived discriminant locus application (Example
\ref{example: ddislocus}).
Efe \.{I}zbudak formulated the statements of the derived
transversality theorem (Theorem \ref{thm:derived transversality}) and
the Legendrian intersection theorem (Theorem \ref{thm:Leg
intersection}), provided the proofs, and constructed the applications
to moduli problems (Section \ref{sec:moduli}).
The authors warmly thank the Higher Structures Research Group at METU
for their support.

\bibliographystyle{amsalpha}
\bibliography{refss}

\end{document}